\documentclass[10pt, oneside]{amsart}

\usepackage[hscale = 0.8, top=0.8in]{geometry}

 \usepackage{float}
 \usepackage{amsmath}      
 \usepackage{amsthm}        
 \usepackage{amssymb}       
 \usepackage{empheq}       
 \usepackage{stackrel}
 \usepackage{braket}
 \usepackage[all]{xy} 
 \usepackage{hyperref}
 \hypersetup{
  colorlinks   = true,          
  urlcolor     = pink,          
  linkcolor    = red, 
  citecolor   = blue            
}

\usepackage{tikz}
\usetikzlibrary{matrix}
\usetikzlibrary{shapes}
\usetikzlibrary{arrows}
\usetikzlibrary{calc,3d}
\usetikzlibrary{decorations,decorations.pathmorphing}
\usetikzlibrary{through}
\tikzset{ext/.style={circle, draw,inner sep=1pt},int/.style={circle,draw,fill,inner sep=2pt},nil/.style={inner sep=1pt}}
\tikzset{exte/.style={circle, draw,inner sep=3pt},inte/.style={circle,draw,fill,inner sep=3pt}}
\tikzset{diagram/.style={matrix of math nodes, row sep=3em, column sep=2.5em, text height=1.5ex, text depth=0.25ex}}
\tikzset{diagram2/.style={matrix of math nodes, row sep=0.5em, column sep=0.5em, text height=1.5ex, text depth=0.25ex}}

 \renewcommand{\theequation}{\thesection.\arabic{equation}}         
 \makeatletter                                                      
    \@addtoreset{equation}{section} 
    \@addtoreset{footnote}{section}

 \def\newprooflikeenvironment#1#2#3#4{%
      \newenvironment{#1}[1][]{%
          \refstepcounter{equation}                                 
          \begin{proof}[{\rm\csname#4\endcsname{#2~\theequation}   
          \@ifnotempty{##1}{\the\thm@notefont \ (##1)}\csname#4\endcsname{.}}]
          \def\qedsymbol{#3}}%
         {\end{proof}}}                                             
 \makeatother                                                       
                                                                    
 \newprooflikeenvironment{definition}{Definition}{$\diamond$}{textbf}
 \newprooflikeenvironment{example}{Example}{$\diamond$}{textbf}     
 \newprooflikeenvironment{remark}{Remark}{$\diamond$}{textbf}       
 \newprooflikeenvironment{notation}{Notaion}{$\diamond$}{textbf}    
                                                                    
 \theoremstyle{plain}                                               
 \newtheorem{theorem}[equation]{Theorem}                            
                                          
 \newtheorem*{lemma*}{Lemma}                                        
 \newtheorem*{theorem*}{Theorem}                                    
                             
 \newtheorem*{proposition*}{Proposition}                            
 \newtheorem{lemma}[equation]{Lemma}                                
 \newtheorem{corollary}[equation]{Corollary}                        
 \newtheorem{proposition}[equation]{Proposition}                    
 \newtheorem{conjecture}[equation]{Conjecture}

 \newcommand{\alg}{\mathrm{Alg}}
 \newcommand{\Alg}{\mathcal{A}\text{\sf lg}}
 \newcommand{\ass}{\mathrm{Ass}}
 \newcommand{\br}{\mathrm{Br}}
 \newcommand{\C}{\mathcal{C}}
 
 \newcommand{\comm}{\mathrm{Comm}}
 \newcommand{\coalg}{\mathrm{CoAlg}}
 \newcommand{\ccomm}{\mathrm{coComm}}
 \newcommand{\coE}{\mathrm{co}\mathbb{E}}
 \newcommand{\colie}{\mathrm{coLie}}
 \newcommand{\ccolie}{\mathrm{coLie}^\theta}
 \newcommand{\conv}{\mathrm{Conv}}
 
 \newcommand{\dg}{\mathrm{dg}}
 \newcommand{\Dg}{\mathbf{dg}}
 \newcommand{\define}{\text{def}}
 \newcommand{\der}{\text{Der}}
 
 \DeclareMathOperator{\en}{End}
 \newcommand{\f}{\mathrm{F}}
 \newcommand{\F}{\mathcal{F}}
 
 \newcommand{\free}{\mathrm{free}}
 \newcommand{\Free}{\mathbf{free}}
 \newcommand{\gr}{\mathrm{gr}}
 \DeclareMathOperator{\harr}{Harr}

 \newcommand{\Ind}{\mathbf{ind}}
 \newcommand{\ind}{\mathrm{ind}}

 \newcommand{\La}{\Lambda}
 \newcommand{\lie}{\mathrm{Lie}}

 \newcommand{\map}{\mathrm{Map}}

 \renewcommand{\nu}{\mathrm{nu}}
 \renewcommand{\O}{\mathcal{O}}
 \newcommand{\oblv}{\mathrm{oblv}}
 \newcommand{\Oblv}{\mathbf{oblv}}
 
 \newcommand{\Om}{\Omega}
 \newcommand{\omp}{{\Omega(\mathrm{co}\mathbb{P}_2\{2\})}}
 \newcommand{\op}{\mathrm{Op}}

 \renewcommand{\P}{\mathcal{P}}
 \newcommand{\PP}{\mathbb{P}}

 \newcommand{\pol}{\mathrm{Pol}}
 \newcommand{\cpol}{\widehat{\mathrm{Pol}}}
 \newcommand{\Pol}{\mathcal{P}\mathrm{ol}}
 \newcommand{\pb}{\rule{.4pt}{5.4pt}\rule[5pt]{5pt}{.4pt}\llap{\hspace{1pt}}}
 
 \newcommand{\prelie}{\mathrm{preLie}}

 \DeclareMathOperator{\rk}{rk}
 
 \newcommand{\sgn}{\mathrm{sgn}}
 \newcommand{\s}{\sigma}
 \newcommand{\sh}{\mathrm{Sh}}
 
 \DeclareMathOperator{\spec}{Spec}

 \newcommand{\sym}{\mathrm{Sym}}
 
 \newcommand{\T}{\mathcal{T}}
 \newcommand{\TT}{\mathbb{T}}
 \newcommand{\triv}{\mathrm{triv}}
 \newcommand{\Triv}{\mathbf{triv}}
 \DeclareMathOperator{\uhom}{\underline{Hom}}

 \newcommand{\coP}{\mathrm{co}\PP}
 \newcommand{\ccoP}{\mathrm{co}\PP^\theta}
 \newcommand{\cucoP}{\mathrm{co}\PP^{\text{cu}}}
 \newcommand{\cuccomm}{\mathrm{coComm}^\mathrm{cu}}
 \newcommand{\tP}{{\widetilde{\mathbb{P}^\mathrm{nu}_2}}}

 \newcommand{\inveq}{\mathrel{\rotatebox[origin=c]{90}{$=$}}}
 
 \newcommand{\invpb}{\mathrel{\rotatebox[origin=c]{270}{$\pb$}}}

 \newcommand{\invertin}{\mathrel{\rotatebox[origin=c]{90}{$\in$}}}

\title{Degenerate Poisson algebras and derived Poisson degeneracy loci}
\author{Grigorii Konovalov}
\thanks{The author is a research intern at the Center for Fundamental Mathematics, grant number FSMG-2023-0013, MIPT, Moscow; and a PhD 
student at Skoltech and at the HSE University, which are also located in Moscow}
\date{March 2023}
\address{\quad}
\email{grisha.v.konovalov@gmail.com}

\begin{document}

\begin{abstract}
This paper originated as an attempt to answer a question: what are the natural derived structures on Poisson degeneracy loci?
We argue that the question could be possibly answered via a construction of differential graded operads that
``naturally'' act on the degeneracy 
loci. For each $m \ge 0$, we suggest what looks like a reasonable condition for a Poisson structure on a commutative differential
graded algebra to be $m$-degenerate, i.e. to ``have rank $\le 2m$''. That condition will turn out to be a universal property of
the operad that controls such Poisson algebras; we denote that operad $\PP_1^{\le m}$. We prove that the operad $\PP_1^{\le m}$ does in 
fact exist, and we write an explicit simplicial resolution of it. The latter, in particular, will allow us to show that
$\PP_1^{\le m}$ sits in non-positive cohomological degrees and to compute $H^0(\PP_1^{\le m})$.
\end{abstract}

\maketitle


\section{Introduction}

Let $X$ be a smooth algebraic variety over a base field $k$, $\mathrm{char}(k) = 0$,
equipped with an algebraic Poisson structure, which is given by a Poisson bivector field
\begin{equation*}
\pi_X \in \Gamma(X, \La^2 T(X)) \:.
\end{equation*}
Among Poisson subschemes in
$(X, \pi_X)$ degeneracy loci are distinguished ones. They could be defined as follows.
\begin{definition}\label{def1}
As a subset in $X$, the m-degeneracy locus
\begin{equation*}
D_{m}(X, \pi_X) = \left\{ x \in X : \rk (\pi^\sharp_{X, x} \colon T^*_x(X) \to T_x(X)) \le 2m \right\}
\end{equation*}
consists of points where rank of $\pi_X$ is not greater than $2m$. (Note that $\pi_X$ is a skew field, so its rank is always even.)
Moreover, as a subscheme in $X$, $D_{m}(X, \pi_X)$ is defined as the vanishing locus of the section
\begin{equation*}
\pi_X^{m+1} \equiv \underbrace{\pi_X \wedge \ldots \wedge \pi_X}_{m+1} \in \Gamma(X, \La^{2m+2}T(X)) \:.
\end{equation*}
It is easily seen that the Poisson structure restricts on $D_m(X, \pi_X)$ for each $m$. The $0$-degeneracy locus $D_0(X, \pi_X)$ is 
also sometimes called Poisson vanishing locus; the restriction of the Poisson structure to the vanishing locus is trivial.
\end{definition}

The geometry of Poisson degeneracy loci was extensively studied in a number of works, including \cite{Bondal}, \cite{Pol}, \cite{PG}.
This paper aims towards contributing to this study from the perspective of homotopical algebra and derived algebraic geometry, 
motivation for which was provided
by a remarkable conjecture, made by A. Bondal, claiming that the degeneracy loci are especially large in some cases.
\begin{conjecture}[A. Bondal, \cite{Bondal}]\label{conj-bondal}
Let $X$ be a complex smooth connected Fano variety, equipped with a Poisson structure $\pi_X$.
Then, for each integer $k \ge 0$ such that $2k < \mathrm{dim}(X)$, degeneracy locus $D_{2k}(X, \pi_X)$ has a component of
dimension $\ge 2k+1$.
\end{conjecture}
It is worth noting that the literature contains some evidence favoring Bondal's Conjecture. In particular, it was proved in 
\cite{Bondal} for $X$ of dimension $3$ and in \cite{PG} for $X$ of dimension $4$.

This suggests that the defining property of the degeneracy locus is somewhat non-transversal, which might mean that there is a
meaningful derived analog of the notion of Poisson degeneracy locus that has better transversality, as it usually happens when
we replace classical notions with their derived versions. From this geometric perspective, the main accomplishment of this paper
is a definition of derived Poisson degeneracy loci as derived Poisson subschemes in $X$ in such a way that they refine the
classical degeneracy loci, see Definition \ref{der-degen-loci-def} and Proposition \ref{der-degen-loci-prop}. In other words,
this paper aims at extending the notion of Poisson degeneracy loci to the field of derived Poisson geometry, foundations for
which were established in a number of works -- \cite{PTVV}, \cite{Melani}, \cite{CPTVV}, \cite{Melani_2018}, \cite{Melani_20181},
to name a few.

Our methodology extensively relies on the theory of algebraic operads. Below, we would like to sketch the main ideas
that will form a core of our constructions.

Let us consider the case of the vanishing locus first.
Let $A$ be a Poisson algebra over the base field $k$.
One sees that the vanishing locus $D_0 \subset \spec(A)$ of the Poisson scheme $\spec(A)$ is 
given by $\spec(A/(\{A, A\})) \subset \spec(A)$, where $(\{A,A\}) \subset A$ denotes the ideal in $A$ generated by the brackets.
In other terms, we consider the Poisson algebra $A$ as an algebra over the Poisson operad
$\PP_1$. Then the commutative algebra $A/(\{A, A\})$ can be identified with the induction of $A$ along the morphism of operads
\begin{equation}\label{augmentation}
\PP_1 \longrightarrow \comm
\end{equation}
which maps the bracket to zero. Moreover, the factor map
\begin{equation*}
A \longrightarrow A/(\{A, A\}) \:,
\end{equation*}
which corresponds to the embedding $D_0 \subset \spec(A)$, can be identified as the unit map
\begin{equation*}
A \longrightarrow \text{oblv}_\comm^{\PP_1}(\text{ind}^\comm_{\PP_1}(A))
\end{equation*}
of the adjuntion
\begin{equation*}
\text{ind}_{\PP_1}^\comm \colon \alg_{\PP_1} \rightleftarrows \alg_\comm : \text{oblv}^{\PP_1}_\comm
\end{equation*}
which corresponds to the morphism of operads (\ref{augmentation}).

This suggests a definition for the derived Poisson vanishing locus as the derived induction along the map (\ref{augmentation}).
Moreover, this suggests the following plan for defining the derived Poisson degeneracy loci: first, for each $m \ge 1$,
define an operad $\PP_1^{\le m}$ together with a morphism
\begin{equation}\label{quasiAugmentation}
\PP_1 \longrightarrow \PP_1^{\le m}.
\end{equation}
Then, by definition, taking the derived m-th degeneracy locus of a Poisson algebra corresponds to taking the derived induction along
the morphism of operads (\ref{quasiAugmentation}). Heuristically, one should think of the operad $\PP_1^{\le m}$ as of an operad whose
algebras are Poisson algebras with the bracket having rank $\le 2m$. In addition, the forgetful functor corresponding to the morphism
(\ref{quasiAugmentation}) forgets this degeneracy datum.

It will be handy for us to construct $\PP_1^{\le m}$ by a universal property, which will be similar to the universal property of 
$\PP_1$ studied by V. Melani, \cite{Melani}. Melani constructed a functor
\begin{equation}\label{multi-derivations}
\mathcal{MD} \colon {\op}_{\comm/} \longrightarrow \lie^\text{gr}_k 
\end{equation}
(we let ourselves slightly simplify Melani's notation because,
at the moment, we only need to communicate a small portion of the information the original notation stores) that
takes an operad $\O$, equipped with a map $\mu \colon \comm \to \O$, to the graded complex of anti-symmetric operations of $\O$ that
are multi-derivations with respect to $\mu$. That complex was equipped with a Lie bracket, which was a restriction of the convolution
Lie braket of $\conv(\ccomm, \O\{-1\})$. In \cite{Melani}, the following characterization of $\PP_1$ was proved.
\begin{theorem*}[Theorem 2.11, \cite{Melani}]
The fiber at $\mu$ of the map of simplicial sets
\begin{equation*}
\uhom_\op (Q(\PP_1), \O) \longrightarrow \uhom_\op (\comm, \O)
\end{equation*}
is isomorphic to
\begin{equation*}
\uhom_{\lie_k^\mathrm{gr}} (Q(k(2)[-1]), \mathcal{MD}(\O)) \simeq \map_{\lie_k^\mathrm{gr}} (k(2)[-1], \mathcal{MD}(\O)) \:,
\end{equation*}
where $k(2)[-1]$ is trivial graded Lie albegra sitting in weight degree 2 and cohomological degree 1, $Q(k(2)[-1])$ and $Q(\PP_1)$ are
some nice resolutions of $k(2)[-1]$ and $\PP_1$ respectively chosen somewhat consistently with respect to each other.
\end{theorem*}
\begin{example}\label{intro-ex1}
Let $A$ be a cofibrant commutative differential graded algebra. In this case the endomorphism operad $\en_A$ is equipped with a map
$\comm \to \en_A$ that corresponds to the commutative structure on $A$, and, as was shown in \cite{Melani},
application of $\mathcal{MD}(-)$ to the object $\en_A \in \op_{\comm/}$ yields the graded Lie algebra of polyvector fields
\begin{equation*}
\bigoplus_{n \ge 0} \sym_A^n (\TT_A[-1]) [1] \:,
\end{equation*}
whose Lie bracket is given by the Schouten bracket and whose grading is given by the symmteric powers.
So in this case Melani's Theorem tells us that Poisson structures on $A$ are in a bijection to the set of maps
\begin{equation*}
k(2)[-1] \longrightarrow \bigoplus_{n \ge 0} \sym_A^n (\TT_A[-1]) [1]
\end{equation*}
in the homotopy category of graded dg Lie algebras. Among other things, any such map induces a map of complexes
\begin{equation*}
k[-1] \longrightarrow \sym^2_A(\TT_A[-1])[1] \:,
\end{equation*}
which gives the underlying bivector of the Poisson structure.
\end{example}
In addition, we would like to note that $\mathcal{MD}(\O)[-1]$ could be equipped with a commutative product, which is 
compatible with the Lie bracket. This makes $\mathcal{MD}(\O)[-1]$ into a graded $\PP_2$-algebra; therefore, we can write
\begin{equation*}
\map_{\lie_k^\mathrm{gr}} (k(2)[-1], \mathcal{MD}(\O)) \simeq \map_{\PP_2^\mathrm{gr}} (k[x], \mathcal{MD}(\O)[-1]) \:,
\end{equation*}
where $k[x]$ is the free graded $\PP_2$-algebra generated by the trivial Lie algebra $k(2)[-1]$ -- so the generator $x$ has weight 
degree 2
and cohomological degree 2. It turns out that the product on $\mathcal{MD}(\O)[-1]$ under the isomorphism of Example \ref{intro-ex1}
becomes the usual product of polyvector fields; in particular, the element
\begin{equation*}
x^{m+1} \in k[x] \longrightarrow \mathcal{MD}(\en_A)[-1] \cong \bigoplus_{n \ge 0} \sym_A^n (\TT_A[-1])
\end{equation*}
corresponds to the $(m+1)$-th skew power of the Poisson bivector.
This suggests to define $\PP_1^{\le m}$ by the formula
\begin{equation}\label{heuristical-def}
\map_{\op_{\comm/}} (\PP_1^{\le m}, \O) \simeq \map_{\PP_2^\mathrm{gr}} (k[x]/(x^{m+1}), \mathcal{MD}(\O)[-1]) \:.
\end{equation}
However, the formula (\ref{heuristical-def}) does not give us a good definition for the following reason: the functor $\mathcal{MD}$
does not  preserve quasi-isomorphisms.

Nonetheless, in Section \ref{sec4.2} -- see Definition \ref{poly-def-pol-gr}, we construct a functor
\begin{equation*}
\pol \colon \op_{\comm/} \longrightarrow \alg^{\text{gr}}_{\PP_2} \:,
\end{equation*}
which is similar to $\mathcal{MD}$ in its essence: for instance, for a commuataive algebra $A$, $\pol(\en_A)$ is
quasi-isomorphic to the derived algebra of polyvector fields of $A$, see Theorem \ref{Calaque-Willwacher-thm}.
We prove in Proposition \ref{main-prop} that the functor $\pol$ preserves
quasi-isomorphisms; therefore, it induces a functor between localizations
\begin{equation}\label{derived-operadic-poly}
\Pol \colon \op_{\comm/}[W_\text{qis}^{-1}] \longrightarrow \alg^\text{gr}_{\PP_2}[W^{-1}_\text{qis}] \:.
\end{equation}
Moreover, in Proposition \ref{main-prop}, we also show that the functor (\ref{derived-operadic-poly}) admits a left-adjoint
\begin{equation*}
\F \colon \alg^\text{gr}_{\PP_2}[W^{-1}_\text{qis}] \longrightarrow \op_{\comm/}[W_\text{qis}^{-1}] \:,
\end{equation*}
which allows us to define the $m$-degenerate Poisson operad $\PP_1^{\le m}$ as $\F(k[x]/(x^{m+1}))$ -- see Definition \ref{main-def};
by that definition, $\PP_1^{\le m}$
satisfies the following universal property:
\begin{equation*}
\begin{split}
\map_{\op_{\comm/}} (\PP_1^{\le m}, \O) &\simeq \map_{\op_{\comm/}} (\F(k[x]/(x^{m+1})), \O) \simeq \\
&\simeq \map_{\PP_2^\mathrm{gr}} (k[x]/(x^{m+1}), \Pol(\O)) \:.
\end{split}
\end{equation*}

Moreover, in Proposition \ref{poisson-operad}, we show that there is an equivalence $\F(k[x]) \simeq \PP_1$ compatible with the maps 
from $\comm$, so that the factor map $k[x] \to k[x]/(x^{m+1})$ induces a map $\PP_1 \to \PP_1^{\le m}$, precomposition with which 
forgets the degeneracy datum of the degenerate Poisson structure.

For a commutative dg-algebra $A$, the space $\map_{\op_{\comm}}(\PP_1^{\le m}, \en_A)$
of $m$-degenerate Poisson structures on $A$ is computed as $\map_{\Alg_{\PP_2}^\gr} (k[x]/(x^{m+1}), \Pol(\en_A))$ with the underlying
Poisson structure given by the map
\begin{equation*}
\map_{\Alg_{\PP_2}^\gr} (k[x]/(x^{m+1}), \Pol(\en_A)) \longrightarrow \map_{\Alg_{\PP_2}^\gr} (k[x], \Pol(\en_A)) \simeq \map_{\op_{\comm/}} (\PP_1, \en_A) \:.
\end{equation*}
Let us add that in this case there is an equivalence of graded $\PP_2$-algebras
\begin{equation*}
\Pol(\en_A) \simeq \bigoplus_{n \ge 0} \sym_A^n (\TT_A[-1]) \:;
\end{equation*}
in particular, for the underlying Poisson structure we still have
\begin{equation*}
\map_{\Alg_{\PP_2}^\gr} (k[x], \Pol(\en_A)) \simeq \map_{\Alg_{\lie}^\gr} (k(2)[-1], \Pol(\en_A)[1]) \simeq \map_{\Alg_{\lie}^\gr} (k(2)[-1], \bigoplus_{n \ge 0} \sym_A^n (\TT_A[-1])[1]) \:.
\end{equation*}

\

\smallskip

\noindent \textbf{Notation and conventions.}

\begin{enumerate}

\item[$\bullet$] We work over a base field $k$ of characteristic zero.

\item[$\bullet$] Given a relative category $(C, W)$, we denote the underlying $\infty-$category by $C[W^{-1}]$.

\item[$\bullet$] We denote by $\dg_k$ the categry of chain complexes over k.

\item[$\bullet$] The localized category $\dg_k[W^{-1}_\text{qis}]$ is denoted by $\Dg_k$.

\item[$\bullet$] $\comm$ denotes the operad in chain complexes whose algebras are unital (strictly) commutative algebras in chain 
complexes. In particular, the underlying symmetric sequence is given by
\begin{equation*}
\comm(n) \cong k, \quad n \ge 0,
\end{equation*}
where all symmetric groups act trivially. The non-unital variant of the commutative operad is denoted by $\comm^\text{nu}$.

\item[$\bullet$] $\lie$ denotes the operad in chain complexes whose algebras are the differential graded Lie algebras.

\item[$\bullet$] $\PP_1$ denotes the operad whose algebras are the unital dg-Poisson algebras; the non-unital
variant is denoted $\PP_1^\text{nu}$.
The underlying symmetric sequences are given by
\begin{equation*}
\PP_1(n) \cong (\comm \circ \lie) (n) \cong \bigoplus_{k\ge 0} \comm(k) \otimes_{S_k} \left( \bigoplus_{i_1 + \ldots + i_k = n} \mathrm{Ind}_{S_{i_1}\times \ldots \times S_{i_k}}^{S_n} \lie(i_1) \otimes_k \ldots \otimes_k \lie(i_k) \right) \:,
\end{equation*}
and
\begin{equation*}
\PP^\text{nu}_1(n) \cong (\comm^\text{nu} \circ \lie) (n) \cong \bigoplus_{k\ge 0} \comm^\text{nu}(k) \otimes_{S_k} \left( \bigoplus_{i_1 + \ldots + i_k = n} \mathrm{Ind}_{S_{i_1}\times \ldots \times S_{i_k}}^{S_n} \lie(i_1) \otimes_k \ldots \otimes_k \lie(i_k) \right) \:,
\end{equation*}
see \cite{LV}, Section 13.3 for a description of the Poisson operad and Section 5.1 for a description of the composition monoidal 
product.

\item[$\bullet$] $\PP_2$ denotes the Gerstenhaber operad; its algebras are (shifted) Poisson algebras whose Lie bracket has 
cohomological degree $1$. An example of a $\PP_2$-algebra is the algebra of polyvector fields of a commutative algebra $A$
\begin{equation*}
\bigoplus_{n \ge 0} \sym_A^n (\TT_A[-1])
\end{equation*}
equipped with the Schouten bracket. The non-unital variant of the operad is denoted by $\PP_2^\text{nu}$.

\item[$\bullet$] At some point, we will need to work with $\lie$ and $\PP_2^\text{nu}$ algebras in chain complexes equipped with
additional grading. We will adopt a standard convention that the Lie bracket has weight -1 with respect the grading. We will use a
superscript $\mathrm{gr}$ to emphasize the fact that we are working with graded algebras.
For example, the algebra of polyvector fields 
of a commutative algebra $A$
\begin{equation*}
\bigoplus_{n \ge 0} \sym_A^n (\TT_A[-1])
\end{equation*}
is graded by the symmetric powers, so we write $\bigoplus_{n \ge 0} \sym_A^n (\TT_A[-1]) \in \alg^\gr_{\PP^\text{nu}_2}$.

\item[$\bullet$] We will also be interested in graded $\Omega(\coP_2\{2\})$-algebras. Operad $\Omega(\coP_2\{2\})$ is a quasi-free
model of the operad $\PP_2^\nu$, it is generated by the symmetric sequence $\overline{\coP_2\{2\}}[-1]$ and equipped with the cobar 
differential. We refer to Section \ref{subsec23} and Section \ref{subsec24} for some details and references
regarding the definition of the operad $\Omega(\coP_2\{2\})$. The weight of a generator, which is an element of
$\overline{\coP_2\{2\}}[-1]$, is defined to be equal to 1 minus the arity of the generator's $\ccomm$-part. In other words,
according to this definition, the Lie bracket has weight -1, the product has weight 0, weights of the other generators are defined
accordingly.

\item[$\bullet$] In Example \ref{computation1} and Example \ref{xsquared}, we will need to perform computations which involve
permutations. We would like to fix the following notation for the permutations: an element $\s \in S_n$ is denoted by the unique
$n-$tuple $(i_1, i_2, \ldots, i_n)$ such that
\begin{equation*}
(i_1, i_2, \ldots, i_n) = \s \cdot (1, 2, \ldots, n) \:,
\end{equation*}
where the right-hand side involves the usual action of $S_n$ on $n$-tuples.
For example, the transposition permuting the first two elements is denoted by $(2, 1, 3, 4, \ldots, n)$.

\item[$\bullet$] $\sh(i,j) \subset S_n$ is the subset of $(i,j)$-shuffles, it consists of permutations $\s \in S_n$ satisfying
\begin{equation*}
\s(1) < \s(2) < \ldots < \s(i) \:,
\end{equation*} and
\begin{equation*}
\s(i+1) < \s(i+2) < \ldots < \s(i+j = n) \:.
\end{equation*}
For example, the permutation $(1, 3, 2, 4) \in S_4$ is a $(2,2)$-shuffle; the permutation $(2, 1, 3, 4) \in S_4$ is not a 
$(2,2)$-shuffle.

\end{enumerate}

\

\smallskip

\noindent \textbf{Structure of the paper.}

In Section \ref{sec2}, we recall a few standard facts from the theory of algebraic operads.
In Section \ref{sec3}, we give a detailed look at the derived vanishing locus, defined as the derived induction along the map of 
operads (\ref{augmentation}). Specifically, we provide a general formula for the derived vanishing locus and we compute it in one 
particular 
case, in which the vanishing locus can be computed explicitly. In Section \ref{sec4}, we define and study the functor of polyvectors
$\pol \colon \op_{\comm/} \longrightarrow \alg^\text{gr}_{\PP_2}$, which then allows us to define the degenerate Poisson operads.
In addition, for each $m > 0$, we construct a simplicial resolution of the $m$-degenerate Poisson operad $\PP_1^{\le m}$
and compute $H^0(\PP_1^{\le m})$.
Finally, in Section \ref{sec5}, we answer a few basic questions about the degeneracy loci regarding their relation to the newly 
defined degenerate Poisson operads. Specifically, we prove that the classical induction along the map
$\PP_1 \to H^0(\PP_1^{\le m})$ 
gives us the classical degeneracy locus as defined in \ref{def1}, and that $H^0(-)$ of the derived degeneracy locus is given by the
classical degeneracy locus.

\

\smallskip

\noindent \textbf{Acknowledgements.} The author is indebted to Alexey Bondal and Christopher Brav for their constant support
and numerous helpful discussions throughout the work on this paper, and, in particular, to Christopher Brav for reading the draft
and making a few valuable suggestions. Also, the author kindly acknowledges that he benefited from
a few conversations with Alexey Rosly and Artem Prikhodko. Besides, the author would like to thank Alexey Bondal for introducing
him to the Poisson degeneracy loci and, in particular, to Conjecture \ref{conj-bondal}, which inspired many of the ideas 
expressed in this paper. The project was partially supported by RFBR and CNRS, grant number 21-51-15005.



\section{Basic facts and definitions}\label{sec2}

\subsection{Operads}

In this paper, we work with differential graded operads, i.e. operads in chain complexes over the base field $k$.
By an operad we always mean a dg-operad; we denote the category of operads by $\op$. Our definitions and notations for operads
mostly follow those of \cite{LV}.

Given a symmetric sequence $M \in \text{SymSeq}$, its shift $M[n] \in \text{SymSeq}$ is defined by the formula
\begin{equation*}
(M[n])(m) = M(m)[n] \:.
\end{equation*}
Note that, even if there was an operad structure on the symmetric sequence $M$, there are no canonical operad structure on $M[n]$ for
non-zero $n$ -- in fact, there could be no operad structure on $M[n]$ at all. However, given a symmetric sequence $M$, another symmetric
sequence $M\{n\} \in \text{SymSeq}$ is defined by the formula
\begin{equation*}
M\{n\} = M \stackrel{\text{H}}{\otimes_k} \en_{k[-n]} \:,
\end{equation*}
or, equivalently,
\begin{equation*}
M\{n\}(m) = M(m) \otimes_k \sgn_m^{\otimes n} [n(m-1)] \:,
\end{equation*}
where $\sgn_m$ denotes the sign representation of the symmetric group $S_m$; an operad (cooperad) structure on $M$ canonically induces
one on $M\{n\}$, see \cite{LV}, Section 5.3.2.

\subsection{Algebras over an operad}

Given an operad $\O$, we denote by $\alg_\O$ its category of algebras in chain complexes. Similarly, for a cooperad $\C$, we denote
by $\coalg_\C$ its category of conilpotent coalgebras in chain complexes. A map of operads
\begin{equation*}
\O_1 \longrightarrow \O_2
\end{equation*}
induces a Quillen adjuntion
\begin{equation*}
\ind^{\O_2}_{\O_1} \colon \alg_{\O_1} \rightleftarrows \alg_{\O_2} : \oblv^{\O_1}_{\O_2} \:,
\end{equation*}
where the right adjoint, denoted by $\oblv^{\O_1}_{\O_2}$, is the forgetful functor. This Quillen adjunction induces an adjunction
between the underlying $\infty-$categories:
\begin{equation*}
\Ind^{\O_2}_{\O_1} \colon \Alg_{\O_1} = \alg_{\O_1}[W^{-1}_\text{qis}] \rightleftarrows \alg_{\O_2}[W^{-1}_\text{qis}] = \Alg_{\O_2} : \Oblv^{\O_1}_{\O_2} \:,
\end{equation*}
where $W_\text{qis}$ denotes a class of morphisms which are quasi-isomorphisms. In a special case where $\O_1 = \mathbf{1}$ is the
trivial operad, whose category of algebras $\alg_{\mathbf{1}}$ is the category of complexes $\dg_k$, we adopt the following notation:
\begin{equation*}
\free^{\O_2} \colon \dg_k \rightleftarrows \alg_{\O_2} : \oblv_{\O_2}
\end{equation*}
denotes the Quillen adjunction, and
\begin{equation*}
\Free^{\O_2} \colon \Dg_k \rightleftarrows \Alg_{\O_2} : \Oblv_{\O_2}
\end{equation*}
denotes the adjuntion between $\infty-$categories.

It is a standard fact that $\Oblv^{\O_1}_{\O_2}$ is conservative and preserves sifted colimits, see, for example, \cite{GR2},
Chapter 6, Section 1.1.

\subsection{Koszul duality}\label{subsec23}

Recall that an operad $\O$ is called augmented if it is equipped with a map of operads
\begin{equation*}
\O \longrightarrow \mathbf{1} \:,
\end{equation*}
where $\mathbf{1} \in \op$ is the trivial operad. Typically, operads of interest to us that control non-unital algebras have
$\O(0) = 0$ and $\O(1) \cong k$; such an operad admits a unique augmentation. Examples include $\ass^\text{nu}$, $\comm^\text{nu}$, 
$\lie$, and $\PP_1^\text{nu}$.

We refer to \cite{LV}, Section 7, for the Koszul duality theory for augmented operads. Given a quadratic augmented operad $\O$, the
theory produces a quadratic coaugmented cooperad $\C$ together with a morphism
\begin{equation}\label{basics-kos-res}
\Omega \C \longrightarrow \O \:,
\end{equation}
where $\Omega$ denotes the operadic cobar construction, see \cite{LV}, Section 6.5, which maps a coaugmented cooperad into
a quasi-free operad generated by $\overline{\C}[-1]$ whose differential comes from the differential of $\C$ and the
coproduct of $\C$. Operad $\O$ is called Koszul if the map (\ref{basics-kos-res}) is a quasi-isomorphism; in that case
the cooperad $\C$ is said to be Koszul dual to the operad $\O$, and 
(\ref{basics-kos-res}) gives us a quasi-free resolution of the operad $\O$. All of the concrete operads mentioned in the previous 
paragraph are known to be Koszul.

Given a Koszul dual pair $\C - \O$, there is an adjoint pair
\begin{equation*}
\Omega_{\O} \colon \coalg_\C \rightleftarrows \alg_\O : \text{B}_{\O} \:.
\end{equation*}
The right adjoint is called bar construction, it maps an algebra $A \in \alg_\O$ to a quasi-cofree coalgebra
\begin{equation*}
\bigoplus_{n\ge 0} (\C(n) \otimes_k A^{\otimes n})^{S_n}
\end{equation*}
equipped with the bar differential, that is combined from the differential on $A$ and the algebra structure on $A$, see \cite{LV},
Section 11. Similarly, $\Omega_\O$ maps a coalgebra $C \in \coalg_\C$ to a quasi-free algebra
\begin{equation*}
\bigoplus_{n\ge 0} (\O(n) \otimes_k C^{\otimes n})_{S_n}
\end{equation*}
equipped with the cobar differential. Recall that a morphism in $\coalg_\C$ is defined to be a weak-equivalence if the cobar 
construction
$\Omega_\O$ maps it into a quasi-isomorphism. Such weak-equivalences form a subcategory $W_\text{Kos} \subset \coalg_\C$.
It turns out that, for any algebra, $A \in \alg_\O$, the adjunction map
\begin{equation*}
\Omega_\O \text{B}_\O (A) \longrightarrow A
\end{equation*}
is a quasi-isomorphism; therefore, functors $\Omega_\O$ and $\text{B}_\O$ preserve weak-equivalences, and the resulting
adjunction
\begin{equation*}
\Omega_{\O} \colon \coalg_\C[W^{-1}_\text{Kos}] \rightleftarrows \alg_\O[W^{-1}_\text{qis}] = \Alg_\O : \text{B}_{\O}
\end{equation*}
is an equivalence of $\infty-$categories, see, for example, \cite{Safronov_2018}, Proposition 1.7.

It happens that opeards that control unital algebras, such as $\ass$, $\comm$, and $\PP_1$, tend to not admit an augmentation. Yet,
developed in \cite{curved_Koszul}, there is a Koszul duality theory for semi-agumented operads: an operad $\O$ is called semi-augmented
if it is equipped with a map of the underlying symmetric sequences
\begin{equation*}
\O \longrightarrow \mathbf{1}
\end{equation*}
not necessarily commuting with the differentials, such that the composite
\begin{equation*}
\mathbf{1} \longrightarrow \O \longrightarrow \mathbf{1}
\end{equation*}
is identity. The Koszul duality for semi-augmented operads works identically to the Koszul duality for augemented operads, except the
differential of the Koszul dual cooperad $\C$ does not squares to zero, and a curving $\C(1) \longrightarrow k[2]$ is introduced
to account for that, see \cite{curved_Koszul}, Section 3.

\subsection{Examples}\label{subsec24}

First, we will need the following fact, which has become classics by now and is found in \cite{LV}, Section 13,
that operads $\comm^\text{nu}$, $\lie$, $\PP^\text{nu}_1, \PP^\nu_2$ are Koszul, and their Koszul dual are $\colie\{1\}$,
$\ccomm\{1\}$, $\coP_1\{1\}$, and $\coP_2\{2\}$ respectively.

It follows from \cite{curved_Koszul}, Proposition 4.3.3, that operads $\comm$ and $\PP_1$ are also Koszul. Now, in analogy with
the case of the associative operad, which was studied in \cite{curved_Koszul}, Section 6, we would like to sketch a description of the
Koszul dual of $\comm$; it is usually denoted by $\ccolie\{1\}$. In accordance with \cite{curved_Koszul}, Section 4, the underlying
graded cooperad of $\ccolie\{1\}$ is
\begin{equation*}
\colie\{1\} \oplus \coE_0\{-1\} \:,
\end{equation*}
where $\coE_0$ is the cooperad of complexes equipped with a functional and $\oplus$ denotes the product in the category of conilponent
cooperads, i.e. the conilpotent cooperad cogenerated by $\colie\{1\}$ and $\coE_0\{-1\}$. Moreover, $\ccolie\{1\}$ is equipped with 
a trivial differential and a curving
\begin{equation*}
\theta \colon \ccolie\{1\}(1) \longrightarrow k[2]
\end{equation*}
that sends the cooperation
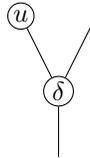
\begin{equation*}
\begin{tikzpicture}[baseline = -5]
\node (v1) at (0.5,1) {};
\node[ext] (v2) at (-0.5,1) {$u$};
\node[ext] (v3) at (0, 0) {$\delta$};
\node (vr) at (0, -1) {};
\draw (v1) edge (v3);
\draw (v2) edge (v3);
\draw (v3) edge (vr);
\end{tikzpicture}
\end{equation*}
to $-1$. A coalgebra over $\ccolie$ is a Lie coalgebra $C$ equipped with a degree $1$ derivation $d \colon C \to C$ and a curving
$\theta \colon C \to k[2]$ that satisfy relations
\begin{equation*}
d^2x = \theta(x_{(1)})x_{(2)}
\end{equation*} and
\begin{equation*}
\theta(dx) = 0 \:.
\end{equation*}
Coalgebras over $\ccolie$ are called curved Lie coalgebras.

The Koszul dual of $\PP_1$ admits a similar describtion, see \cite{Safronov_2018}, Section 1.4. It is denoted by $\ccoP_1\{1\}$, and its
coalgebras are shifted curved Poisson coalgebras.

\subsection{Hopf cooperads}

Recall that the category of cooperads has a symmetric monoidal structure given by the Hadamard product. A Hopf cooperad is a
commutative algebra in cooperads. For example, the cocomuutative counital cooperad $\ccomm^\text{cu}$ is the monoidal unit for
that symmetric monoidal structure; thus, the cooperad $\ccomm^\text{cu}$ is canonically a Hopf cooperad. Another example, which is
of a particular interest to us, is that setting
\begin{equation*}
\delta * \delta  = 0 \:,
\end{equation*}
where $\delta \in \cucoP_1(2)$ is the Lie cobracket, makes the cooperad $\cucoP_1$ into a Hopf cooperad.

\subsection{Brace construction}\label{basics-brace}

Let $\C$ be a cooperad, and $\O$ be an operad. Recall that one can form the convolution pre-Lie algebra
\begin{equation}\label{basics-conv-alg}
\conv(\C, \O) = \prod_{n \ge 0} \uhom_{S_n} (\C(n), \O(n))
\end{equation}
whose pre-Lie product is given by the formula
\begin{equation*}
f \star g = \gamma_{(1)} \circ (f \otimes g) \circ \Delta_{(1)} \:,
\end{equation*}
where $\gamma_{(1)}$ and $\Delta_{(1)}$ denote the infinitesimal composition and the infinitesimal decomposition of $\O$ and $\C$ 
correspondingly, see \cite{LV}, Section 6.4.2.

According to \cite{prelie}, the operad $\prelie$ that controls pre-Lie algebras has the following description:
$\prelie(n)$ is the vector space spanned by rooted trees with $n$ vertices labeled by numbers from $1$ to $n$, and the symmetric group 
$S_n$ acts on $\prelie(n)$ by permutation of the labels. Given a pair of trees $T \in \prelie(n)$ and $S \in \prelie(m)$, their
composition $T \circ_i S$ is defined by the formula
\begin{equation*}
T \circ_i S = \sum_{f \colon \text{In}(T, i) \to [m]} T \circ_i^f S \:,
\end{equation*}
where $\text{In}(T, i)$ denotes the set incoming edges at the vertex $i$, and $T \circ_i^f S \in \prelie(n+m-1)$ denotes the rooted
tree obtained by replacing the vertex $i$ of the tree $T$ by the tree $S$ and attaching edges from $\text{In}(T, i)$ via the map $f$;
the root of $T\circ_i^f S$ is given by the of root of $T$ in case that $i$ is not the root of $T$ and by the root of $S$ otherwise.
For example,
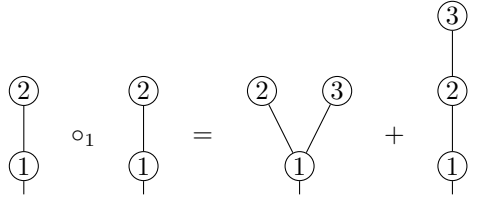
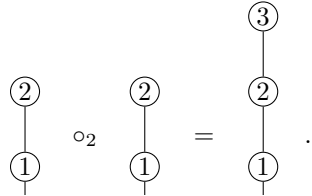
\begin{equation*}
\begin{tikzpicture}[baseline = -5]
\node[ext] (v1) at (0,-0.5) {$1$};
\node[ext] (v2) at (0,0.5) {$2$};
\node (v3) at (0, -1) {};
\draw (v1) edge (v2);
\draw (v1) edge (v3);
\end{tikzpicture}
\quad \circ_1 \quad
\begin{tikzpicture}[baseline = -5]
\node[ext] (v1) at (0,-0.5) {$1$};
\node[ext] (v2) at (0,0.5) {$2$};
\node (v3) at (0, -1) {};
\draw (v1) edge (v2);
\draw (v1) edge (v3);
\end{tikzpicture}
\quad = \quad
\begin{tikzpicture}[baseline = -5]
\node[ext] (v1) at (0,-0.5) {$1$};
\node[ext] (v2) at (-0.5,0.5) {$2$};
\node[ext] (v3) at (0.5,0.5) {$3$};
\node (vr) at (0, -1) {};
\draw (v1) edge (v2);
\draw (v1) edge (v3);
\draw (v1) edge (vr);
\end{tikzpicture}
\quad + \quad
\begin{tikzpicture}[baseline = -5]
\node[ext] (v1) at (0,-0.5) {$1$};
\node[ext] (v2) at (0,0.5) {$2$};
\node[ext] (v3) at (0,1.5) {$3$};
\node (vr) at (0, -1) {};
\draw (v1) edge (v2);
\draw (v2) edge (v3);
\draw (v1) edge (vr);
\end{tikzpicture}
\quad,
\end{equation*}
and
\begin{equation*}
\begin{tikzpicture}[baseline = -5]
\node[ext] (v1) at (0,-0.5) {$1$};
\node[ext] (v2) at (0,0.5) {$2$};
\node (v3) at (0, -1) {};
\draw (v1) edge (v2);
\draw (v1) edge (v3);
\end{tikzpicture}
\quad \circ_2 \quad
\begin{tikzpicture}[baseline = -5]
\node[ext] (v1) at (0,-0.5) {$1$};
\node[ext] (v2) at (0,0.5) {$2$};
\node (v3) at (0, -1) {};
\draw (v1) edge (v2);
\draw (v1) edge (v3);
\end{tikzpicture}
\quad = \quad
\begin{tikzpicture}[baseline = -5]
\node[ext] (v1) at (0,-0.5) {$1$};
\node[ext] (v2) at (0,0.5) {$2$};
\node[ext] (v3) at (0,1.5) {$3$};
\node (vr) at (0, -1) {};
\draw (v1) edge (v2);
\draw (v2) edge (v3);
\draw (v1) edge (vr);
\end{tikzpicture}
\quad.
\end{equation*}
Under this description, the pre-Lie product corresponds to the tree
$\scriptstyle
\begin{tikzpicture}[scale=.85,baseline=.5]
\node[ext] (v1) at (0,-0.5) {$1$};
\node[ext] (v2) at (0,0.5) {$2$};
\node (v3) at (0, -1) {};
\draw (v1) edge (v2);
\draw (v1) edge (v3);
\end{tikzpicture}
$ $\in \prelie(2)$.

Now let $\C$ be a Hopf cooperad such that $\C(0) \cong \C(1) \cong k$. In this case, D. Calaque and T. Willwacher, \cite{Calaque_2015},
extended the action of the operad $\prelie$ on the convolution algebra (\ref{basics-conv-alg}) to an action of an operad $\prelie_\C$,
whose definition we now would like to sketch. The vector space $\prelie_\C(n)$ is spanned by the same rooted trees that span
$\prelie(n)$, but now each vertex of a tree is additionally labeled by an element of $\C$ of arity equal to the in-degree of the
vertex. The composition coincides with that of $\prelie$ on the underlying trees and uses cooperad and Hopf structures on $\C$ to
act on the labels, see \cite{Calaque_2015}, Section 3. The operad $\prelie_\C$ is generated by labeled trees of the form
\begin{equation}\label{basics-prelie-generator}
\begin{tikzpicture}[baseline = -5]
\node[ext] (v2) at (-1,0.5) {$1$};
\node[ext] (v3) at (1,0.5) {$k$};
\node (v0) at (0,-0.5) {};
\node (vv) at (0,0.5) {$\cdots$};
\node[ext] (v1) at (0,0) {$0$};
\draw (v0) edge (v1);
\draw (v2) edge (v1);
\draw (v3) edge (v1);
\end{tikzpicture} \quad.
\end{equation}
Action of the generators of $\prelie_\C$ on the convolution algebra
\begin{equation*}
\conv(\C, \O) = \prod_{n \ge 0} \uhom_{S_n} (\C(n), \O(n))
\end{equation*}
has the following description: tree (\ref{basics-prelie-generator}) labeled by an element $c \in \C(k)$ maps a tuple
$(f, f_1, f_2, \ldots, f_k)$, $f \in \uhom_{S_{k+j}}(\C(k+j), \O(k+j))$, $f_i \in \uhom_{S_{j_i}}(\C(j_i), \O(j_i))$, into
\begin{equation*}
\begin{tikzpicture}[baseline = -5]
\node[ext] (v2) at (-1,0.5) {$1$};
\node[ext] (v3) at (1,0.5) {$k$};
\node (v0) at (0,-0.5) {};
\node (vv) at (0,0.5) {$\cdots$};
\node[ext] (v1) at (0,0) {$0$};
\draw (v0) edge (v1);
\draw (v2) edge (v1);
\draw (v3) edge (v1);
\end{tikzpicture} \quad (f (\tilde{c} \cdot -) , f_1, f_2, \ldots, f_k) \:,
\end{equation*}
where $\tilde c$ denotes the image of $c$ under the map
\begin{equation*}
\C(k) \longrightarrow \C(k+j) \otimes_k \underbrace{\C(1) \otimes_k \ldots \otimes_k \C(1)}_k \otimes_k \underbrace{\C(0) \otimes_k \ldots \otimes_k \C(0)}_j \cong \C(k+j) \:,
\end{equation*}
$f(\tilde{c} \cdot -)$ denotes the compostion of multiplication by $\tilde c$ and $f$, and the action of the tree on the tuple is given 
by the underlying action of $\prelie$.

Given a Maurer-Cartan element $\mu$ in a dg-Lie algebra, the same Lie algebra equipped with a new differential $d + [\mu, -]$ is, in
fact, a dg-Lie algebra. However, a pre-Lie structure in general does not descend to a pre-Lie structure on the twisted Lie algebra,
let alone a $\prelie_\C$-structure. Nevertheless, V. Dolgushev and T. Willwacher, \cite{Twisting}, introduced a construction called
operadic twisting, which, in particular, when applied to $\prelie_\C$, tells us which operad acts on the convolution algebra
$\conv(\C, \O)$ twisted by a Maurer-Cartan element.

Now, following \cite{Calaque_2015}, we would like to describe that operad, denoted $\br_\C$, together with its action on the twisted
convolution algebra. We will focus on a special case where $\C = \cucoP_1$, which is of a special interest to us.
Operad $\br_{\cucoP_1}$ is generated by trees
\begin{equation*}
\begin{tikzpicture}[baseline = -5]
\node[ext] (v2) at (-1,0.5) {$1$};
\node[ext] (v3) at (1,0.5) {$k$};
\node (v0) at (0,-0.5) {};
\node (vv) at (0,0.5) {$\cdots$};
\node[ext] (v1) at (0,0) {$0$};
\draw (v0) edge (v1);
\draw (v2) edge (v1);
\draw (v3) edge (v1);
\end{tikzpicture} \quad\quad
\begin{tikzpicture}[baseline = -5]
\node[ext] (v2) at (-1,0.5) {$1$};
\node[ext] (v3) at (1,0.5) {$k$};
\node (v0) at (0,-0.5) {};
\node (vv) at (0,0.5) {$\cdots$};
\node[int] (v1) at (0,0) {};
\draw (v0) edge (v1);
\draw (v2) edge (v1);
\draw (v3) edge (v1);
\end{tikzpicture}
\qquad,
\end{equation*}
where the former tree is labeled by an element from $\cucoP_1(k)$, and the latter tree has at least $2$ leaves and
is labeled by an element from $\overline{\coP_1}(k)[-1]$. The latter tree has arity $k$; in terms of the underlying
$\prelie_\C$-action on the underlying graded vector space, its action on a tuple
$(f_1, \ldots, f_k)$, $f_i \in \uhom_{S_{j_i}} (\cucoP_1\{1\}(j_i), \O(j_i))$, is given by
\begin{equation*}
\begin{tikzpicture}[baseline = -5]
\node[ext] (v2) at (-1,0.5) {$1$};
\node[ext] (v3) at (1,0.5) {$k$};
\node (v0) at (0,-0.5) {};
\node (vv) at (0,0.5) {$\cdots$};
\node[ext] (v1) at (0,0) {$0$};
\draw (v0) edge (v1);
\draw (v2) edge (v1);
\draw (v3) edge (v1);
\end{tikzpicture} \quad (\mu , f_1, f_2, \ldots, f_k) \:,
\end{equation*}
where the root of the tree is labeled by the same element from $\overline{\coP_1}(k)[-1]$, and $\mu$ denotes the Maurer-Cartan
element. Operadic composition of $\br_{\cucoP_1}$ is similar to that of $\prelie_\C$. Differential on $\br_{\cucoP_1}$
is roughly given by the following formulas:
\begin{equation*}
d \begin{tikzpicture}[baseline = -5]
\node[ext] (v2) at (-1,0.5) {$1$};
\node[ext] (v3) at (1,0.5) {$k$};
\node (v0) at (0,-0.5) {};
\node (vv) at (0,0.5) {$\cdots$};
\node[ext] (v1) at (0,0) {$0$};
\draw (v0) edge (v1);
\draw (v2) edge (v1);
\draw (v3) edge (v1);
\end{tikzpicture} = \sum (
\pm\begin{tikzpicture}[baseline = -5]
\node (empty) at (0, -0.5) {};
\node[int] (root) at (0,0) {};
\node[ext] (v1) at (-1, 0.5) {$1$};
\node[ext] (vc) at (0.4, 0.5) {$0$};
\node[ext] (vr) at (1.4, 0.5) {$k$};
\node[ext] (vcl) at (-0.1, 1) {$i$};
\node[ext] (vcr) at (0.9, 1) {$j$};
\node (dots1) at (-0.3, 0.5) {$\cdots$};
\node (dots2) at (0.4, 1) {$\cdots$};
\node (dots3) at (0.9, 0.5) {$\cdots$};
\draw (empty) edge (root);
\draw (root) edge (v1);
\draw (root) edge (vc);
\draw (root) edge (vr);
\draw (vc) edge (vcr);
\draw (vc) edge (vcl);
\end{tikzpicture} \pm 
\begin{tikzpicture}[baseline = -5]
\node (empty) at (0, -0.5) {};
\node[ext] (root) at (0,0) {$0$};
\node[ext] (v1) at (-1, 0.5) {$1$};
\node[int] (vc) at (0.4, 0.5) {};
\node[ext] (vr) at (1.4, 0.5) {$k$};
\node[ext] (vcl) at (-0.1, 1) {$i$};
\node[ext] (vcr) at (0.9, 1) {$j$};
\node (dots1) at (-0.3, 0.5) {$\cdots$};
\node (dots2) at (0.4, 1) {$\cdots$};
\node (dots3) at (0.9, 0.5) {$\cdots$};
\draw (empty) edge (root);
\draw (root) edge (v1);
\draw (root) edge (vc);
\draw (root) edge (vr);
\draw (vc) edge (vcr);
\draw (vc) edge (vcl);
\end{tikzpicture} ) \qquad,
\end{equation*}
and
\begin{equation*}
d \begin{tikzpicture}[baseline = -5]
\node[ext] (v2) at (-1,0.5) {$1$};
\node[ext] (v3) at (1,0.5) {$k$};
\node (v0) at (0,-0.5) {};
\node (vv) at (0,0.5) {$\cdots$};
\node[int] (v1) at (0,0) {};
\draw (v0) edge (v1);
\draw (v2) edge (v1);
\draw (v3) edge (v1);
\end{tikzpicture} = \sum \pm
\begin{tikzpicture}[baseline = -5]
\node (empty) at (0, -0.5) {};
\node[int] (root) at (0,0) {};
\node[ext] (v1) at (-1, 0.5) {$1$};
\node[int] (vc) at (0.4, 0.5) {};
\node[ext] (vr) at (1.4, 0.5) {$k$};
\node[ext] (vcl) at (-0.1, 1) {$i$};
\node[ext] (vcr) at (0.9, 1) {$j$};
\node (dots1) at (-0.3, 0.5) {$\cdots$};
\node (dots2) at (0.4, 1) {$\cdots$};
\node (dots3) at (0.9, 0.5) {$\cdots$};
\draw (empty) edge (root);
\draw (root) edge (v1);
\draw (root) edge (vc);
\draw (root) edge (vr);
\draw (vc) edge (vcr);
\draw (vc) edge (vcl);
\end{tikzpicture} \qquad;
\end{equation*}
on the labels the differential acts as the cobar differential, see \cite{Twisting}, Section 8.

Drawing inspiration from D. Tamarkin's work \cite{Tamarkin}, D. Calaque and T. Willwacher, \cite{Calaque_2015}, Section 3.2, 
constructed a map
\begin{equation}\label{basics-tamarkin-map}
\Omega(\coP_2\{1\}) \longrightarrow \br_{\cucoP_1} \:.
\end{equation}
The map (\ref{basics-tamarkin-map}) acts on generators as follows:
\begin{enumerate}
\item [1)] a generator $\underline{X_1 \ldots X_k} \in \colie(k) \subset \coP_2\{1\}(k)$ is send to the tree
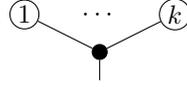
\begin{equation*}
\begin{tikzpicture}[baseline = -5]
\node[ext] (v2) at (-1,0.5) {$1$};
\node[ext] (v3) at (1,0.5) {$k$};
\node (v0) at (0,-0.5) {};
\node (vv) at (0,0.5) {$\cdots$};
\node[int] (v1) at (0,0) {};
\draw (v0) edge (v1);
\draw (v2) edge (v1);
\draw (v3) edge (v1);
\end{tikzpicture}
\end{equation*}
labeled by the element $s^{-1}\underline{X_1 \ldots X_k} \in \colie(k)[-1] \subset \coP_1(k)[-1]$, where by
$\underline{X_1 \ldots X_k}$ we denote the image of the $k$-ary comultiplication under the projection
$\mathrm{coAss}(k) \to \colie(k)$;
\item [2)] the generator $X_1 \wedge X_2 \in \ccomm\{1\}(2) \subset \coP_2\{1\}(2)$ is send to the Lie bracket
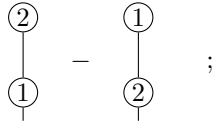
\begin{equation*}
\begin{tikzpicture}[baseline=-5]
\node[ext] (v1) at (0,-0.5) {$1$};
\node[ext] (v2) at (0,0.5) {$2$};
\node (v3) at (0, -1) {};
\draw (v1) edge (v2);
\draw (v1) edge (v3);
\end{tikzpicture} \quad - \quad
\begin{tikzpicture}[baseline=-5]
\node[ext] (v1) at (0,-0.5) {$2$};
\node[ext] (v2) at (0,0.5) {$1$};
\node (v3) at (0, -1) {};
\draw (v1) edge (v2);
\draw (v1) edge (v3);
\end{tikzpicture} \qquad;
\end{equation*}
\item[3)] a generator $X_1 \wedge \underline{X_2 \ldots X_k} \in \coP_2\{1\}(k)$, $k \ge 3$, is send to the tree
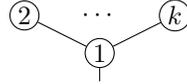
\begin{equation*}
\begin{tikzpicture}[baseline = -5]
\node[ext] (v2) at (-1,0.5) {$2$};
\node[ext] (v3) at (1,0.5) {$k$};
\node (v0) at (0,-0.5) {};
\node (vv) at (0,0.5) {$\cdots$};
\node[ext] (v1) at (0,0) {$1$};
\draw (v0) edge (v1);
\draw (v2) edge (v1);
\draw (v3) edge (v1);
\end{tikzpicture}
\end{equation*}
labeled by the element $s^{-1}\underline{X_2 \ldots X_k} \in \colie(k-1)[-1] \subset \coP_1(k-1)[-1]$;
\item[4)] the rest of the generators are send to zero.
\end{enumerate}
\begin{theorem}[\cite{Calaque_2015}, Theorem 4]
The map (\ref{basics-tamarkin-map}) is a quasi-isomorphism.
\end{theorem}
Let us note that, via the map (\ref{basics-tamarkin-map}), the action of $\br_{\cucoP_1}$ on $\conv(\cucoP_1\{1\}, \O)$ gives rise
to an action of
\begin{equation*}
\Omega(\coP_2\{2\}) \cong \Omega(\coP_2\{1\})\{1\} \:,
\end{equation*}
which is a quasi-free model of $\PP_2^\nu$,
on $\conv(\cucoP_1\{1\}, \O)[-1]$.
\begin{example} [see Section \ref{sec4.1} for details]
For example, if $\O = \en_A$ for a commutative dg-algebra $A$, then the convolution algebra $\conv(\cucoP_1\{1\}, \en_A)[-1]$ twisted
by the commutative structure on $A$ becomes quasi-isomorphic as a (homotopy) $\PP_2^\text{nu}$-algebra to the completed algebra
\begin{equation*}
\prod_{n \ge 0} \sym^n_A(\TT_A[-1])
\end{equation*}
of polyvector fields on $A$.
\end{example}
\begin{remark}
Let us note that we could have started with a non-completed convolution Lie algebra
\begin{equation*}
\bigoplus_{n \ge 0} \uhom_{S_n} (\C(n), \O(n)) \:,
\end{equation*}
and the operad $\prelie_\C$ still acts on this non-completed version. Moreover, we can twist the non-comleted convolution algebra
by an MC element that sits in a finite number of weight degrees, and then the operad $\br_\C$ acts on that twisted non-completed
convolution algebra.
\end{remark}



\section{Derived Poisson vanishing locus}\label{sec3}

In this section we would like to give an explicit description to the functor of vanishing locus\
\begin{equation*}
\mathcal{T}^{\le 0} \colon \Alg_{\PP_1} \longrightarrow \Alg_\comm \:,
\end{equation*}
which is given by derived induction along the natural map of operads
\begin{equation*}
\PP_1 \longrightarrow \comm \:,
\end{equation*}
which is identity on the multiplication and sends the lie bracket to zero.

Recall that, for a Koszul dual pair $\C - \O$, there is a bar-cobar adjunction
\begin{equation*}
\Omega_\O \colon \coalg_\C \rightleftarrows \alg_\O : B_\O \:,
\end{equation*}
which descends to an adjoint equivalence
\begin{equation*}
\Omega_\O \colon \coalg_\C[W^{-1}_\text{Kos}] \stackrel{\simeq}{\rightleftarrows} \alg_\O[W^{-1}_\text{qis}] = \Alg_\O : B_\O \:.
\end{equation*}

\begin{lemma}
Diagram
\[\xymatrix{
\coalg_{\ccoP_1\{1\}}[W^{-1}_\text{Kos}] \ar[rr]^(0.6){\Omega_{\PP_1}}\ar[d]_{\mathbf{oblv}} && \Alg_{\PP_1} \ar[d]^{\mathcal{T}^{\le 0}} \\
\coalg_{\ccolie\{1\}}[W^{-1}_\text{Kos}] \ar[rr]^(0.6){\Omega_\comm} && \Alg_\comm & ,
}\]
where the left vertical arrow is the forgetful functor, is commutative.
\end{lemma}
\begin{proof}
The diagram consists of left adjoints, so it suffices to check that the diagram of their right adjoints
\[\xymatrix{
\coalg_{\ccoP_1\{1\}}[W^{-1}_\text{Kos}] && \Alg_{\PP_1} \ar[ll]^(0.4){B_{\PP_1}} \\
\coalg_{\ccolie\{1\}}[W^{-1}_\text{Kos}] \ar[u]^{\mathbf{cofree}} && \Alg_\comm \ar[ll]^(0.4){B_\comm} \ar[u]^{\mathbf{triv}}
}\]
is also commutative. This diagram is induced by the corresponding diagram of 1-categories
\[\xymatrix{
\coalg_{\ccoP_1\{1\}} && \alg_{\PP_1} \ar[ll]^(0.4){B_{\PP_1}} \\
\coalg_{\ccolie\{1\}} \ar[u]^{\text{cofree}} && \alg_\comm \ar[ll]^(0.4){B_\comm} \ar[u]^{\text{triv}} & ,
}\]
all of whose arrows preserve weak equivalences. So it is enough to check commutativity of the latter, which easily follows from the 
definition of the bar construction.
\end{proof}

\begin{corollary}
Given a Poisson algebra $A \in \Alg_{\PP_1}$, its vanishing locus $\T^{\le 0}(A) \in \Alg_\comm$ is computed as
\begin{equation*}
\Omega_\comm \circ \mathbf{oblv} \circ B_{\PP_1} (A) \:.
\end{equation*}
\end{corollary}
\begin{proof}
\begin{equation*}
\T^{\le 0}(A) \simeq \T^{\le 0} \circ \Omega_{\PP_1} \circ B_{\PP_1} (A) \simeq \Omega_\comm \circ \mathbf{oblv} \circ B_{\PP_1} (A) \:.
\end{equation*}
\end{proof}

Now we would like to consider the following example: for a Lie algebra $\mathfrak{g} \in \Alg_\lie$, let us consider a Poisson algebra 
that $A = \sym_k(\mathfrak{g}) \in \Alg_{\PP_1}$.
We note that there is a commutative diagram
\[\xymatrix{
\sym_k(\mathfrak{g}) \ar@{}[r]|\in & \Alg_{\PP_1} \ar[r] & \Alg_\comm \\
\mathfrak{g} \ar@{|->}[u] \ar@{}[r]|\in & \Alg_\lie \ar[u]\ar[r] & \Alg_\text{Id} \ar[u] \ar@{}[r]|\cong & \Dg_k
}\]
which consists of inductions along the corresponding maps of operads. Therefore, the vanishing locus $\T^{\le 0}(\sym_k(\mathfrak{g}))$
is a free commutative algebra generated by $\Ind^\text{Id}_\lie(\mathfrak{g})$, whose computation is the purpose of the following lemma.
\begin{lemma}
There is a functorial equivalence
\begin{equation*}
\Ind^\mathrm{Id}_\lie(\mathfrak{g}) \cong \overline{\mathrm{C}}_*(\mathfrak{g})[-1] \:,
\end{equation*}
where $\overline{\mathrm{C}}_*(\mathfrak{g})$ is the reduced homological Chevalley-Eilenberg complex of $\mathfrak{g}$.
\end{lemma}
\begin{proof}

It follows from the definitions that there is a commutative square
\[\xymatrix{
\coalg_{\ccomm\{1\}} &&&& \alg_\lie \ar[llll]^{B_\lie} \\
\dg_k \ar[u]^{\text{cofree}} &&&& \dg_k \ar@{}[llll]|= \ar[u]^{\text{triv}} & .
}\]
Passing to the localizations first, and to the left adjoints after, we obtain a commutative square
\[\xymatrix{
\coalg_{\ccomm\{1\}}[W^{-1}_\mathrm{Kos}] \ar[d]^{\Oblv} \ar[rrrr]_\simeq^{\Omega_\lie} &&&& \Alg_\lie \ar[d]^{\Ind^\mathrm{Id}_\lie} \\
\Dg_k &&&& \Dg_k \ar@{}[llll]|= & ,
}\]
from which it follows that
\begin{equation*}
\begin{split}
\Ind^\mathrm{Id}_\lie(\mathfrak{g}) \cong \Ind^\mathrm{Id}_\lie \circ \Omega_\lie \circ B_\lie (\mathfrak{g}) &\cong \Oblv \circ B_\lie(\mathfrak{g}) \\
& \cong \overline{\mathrm{C}}_*(\mathfrak{g})[-1] \:.
\end{split}
\end{equation*}
The last isomorphism $\Oblv \circ B_\lie(\mathfrak{g}) \cong \overline{\mathrm{C}}_*(\mathfrak{g})[-1]$ is basically an avatar of
the definition of the Chevalley-Eilenberg complex.
\end{proof}
So the following holds for the vanishing locus of $\sym_k (\mathfrak{g})$
\begin{corollary}
There is a functorial isomorphism of commutative algebras
\begin{equation*}
\T^{\le 0}(\sym_k(\mathfrak{g})) \cong \sym_k(\overline{\mathrm{C}}_*(\mathfrak{g})[-1]) \:.
\end{equation*}
\end{corollary}



\section{Polyvectors and degenerate Poisson operads}\label{sec4}

\subsection{Polyvectors of a cgda}\label{sec4.1}

Let $A$ be a commutative dg-algebra over the base field $k$. There is a nice exposition in \cite{Melani_2018}, Section 2.2, on how to construct, using Koszul duality, a canonical resolution of $A$ as a mixed graded commutative dg-algebra and extract $A$'s cotangent and
tangent complexes from it, which we now would like to recall.
Let $\colie$ be the cooperad of Lie coalgebras. Recall that its suspension $\colie\{1\}$ is Koszul dual to the operad of non-unital 
commutative algebras $\comm^\text{nu}$. As a commutative algebra, $A^\epsilon$ is defined by the formula
\begin{equation*}
A^\epsilon = \sym^{\ge 1}_k(\colie\{1\}(A)) \:.
\end{equation*}
The grading on $A^\epsilon$ is induced by a grading on $\colie\{1\}$, according to which the weight of $\colie\{1\}(n)$ is set to be 
equal to $1-n$. The mixed differential on $A^\epsilon$ is defined as the bar-cobar differential, see \cite{LV}, Section 11.2.

The counit morphism $\colie\{1\} \longrightarrow \mathbf{I}$ induces a morphism of complexes
\begin{equation*}
\colie\{1\}(A) \longrightarrow A \:,
\end{equation*}
which in turn induces a morphism of graded commutative (non-unital) algebras
\begin{equation*}
\sym_k^{\ge 1}(\colie\{1\}(A)) \longrightarrow A \:.
\end{equation*}
It needs to be checked directly that the latter is in fact a morphism of mixed graded commutative algebras. Then we arrive at a
morphism of (non-unital) cdgas
\begin{equation}\label{canonicalResolution}
|A^\epsilon|^l \longrightarrow A \:.
\end{equation}
It is a standard fact about the bar-cobar resolution that (\ref{canonicalResolution}) is a quasi-isomorphism,
and it was proved in \cite{Melani_2018}, Lemma 2.4, that $|A^\epsilon|^l$ is a cofibrant cdga.

Next we recall the definitions of Harrison chain and cochain complexes, whose realizations model cotangent and tangent complexes of A 
respectively.
\begin{definition}
Harrison chain complex $\harr_\cdot(A,A) \in \text{Mod}_A^{\text{gr}, \epsilon}$ is defined by the formula
\begin{equation*}
\harr_\cdot(A,A) = \Omega^1_{A^\epsilon} \otimes_{A^\epsilon} A \cong A \otimes_k \colie\{1\}(A) \:.
\end{equation*}
Harrison cochain complex $\harr^\cdot(A,A) \in \text{Mod}_A^{\text{gr}, \epsilon}$ is defined by the formula
\begin{equation*}
\harr^\cdot(A,A) = \uhom_{\text{Mod}_A^{\text{gr}, \epsilon}} (\harr_\cdot(A,A), A) \cong \conv(\colie\{1\}, \en_A) \:,
\end{equation*}
where the mixed structure on the convolution algebra could be checked to coincide with the Lie bracket with the MC-element 
corresponding to the  multiplication on $A$.
\end{definition}

Let us note that morphism of graded mixed cdgas
\begin{equation*}
A^\epsilon \longrightarrow A
\end{equation*}
induces a morphism of $A$-modules
\begin{equation}\label{polyvectors::cotangent}
|\Omega^1_{A^\epsilon} \otimes_{A^\epsilon} A|^l \longrightarrow \Omega^1_A \:.
\end{equation}
Proposition 2.6 from \cite{Melani_2018} claims that the latter is a quasi-isomorphism provided that $A$ is a cofibrant cgda.
Applying $\uhom_A(\_, A)$ to the morphism (\ref{polyvectors::cotangent}), we get a chain of quasi-isomorphisms
\begin{equation*}
|\harr^\cdot(A,A)| \cong |\uhom_{\text{Mod}_A^{\text{gr}, \epsilon}} (\harr_\cdot(A,A), A)| \cong \uhom_A (|\harr_\cdot(A,A)|^l, A) \cong \der(A,A) \:.
\end{equation*}
Moreover, it can be checked, see \cite{Melani_2018}, Proposition 2.7, that the resulting quasi-isomorphism
\begin{equation*}
\der(A,A) \cong |\conv(\colie\{1\}, \en_A)|
\end{equation*}
respects the Lie algebra structures.

Now, given a cdga $A$, we want to construct a similar presentation for the graded Poisson algebra of polyvector fields on $A$.
Let us consider a mixed graded $A$-module
\begin{equation*}
A \otimes_k \cucoP_1\{1\}(A)
\end{equation*}
where the mixed structure on $A \otimes_k \cucoP_1\{1\}(A)$ comes from the bar differential on $\cucoP_1\{1\}(A)$ corresponding to
the Poisson structure on $A$ whose bracket is trivial. Since the Lie structure on $A$ is trivial, we get a chain
of isomorphisms of graded mixed complexes
\begin{equation*}
\begin{split}
A \otimes_k \cucoP_1\{1\}(A) &\cong A \otimes_k \cuccomm\{1\} \circ \colie\{1\}(A) \cong A \otimes_k \text{Sym}^{\ge 0}_k ( \colie\{1\}(A)[1] ) [-1] \cong \\
&\cong \text{Sym}^{\ge 0}_A (A \otimes_k \colie\{1\}(A)[1]) [-1] \cong \text{Sym}^{\ge 0}_A ( \harr_\cdot(A,A) [1] ) [-1] \:.
\end{split}
\end{equation*}
Passing to the duals and shifting by 1 to the right, we get an isomorphism of mixed graded complexes
\begin{equation}\label{poly-poly-of-cdga}
\sym_A^{\ge 0} (\harr^\cdot(A,A)[-1]) \cong \conv(\cucoP_1\{1\}, \en_A) [-1] \:,
\end{equation}
where the mixed structure on the RHS corresponds to the commutative algebra strucutre on $A$.

\begin{theorem}[Higher Hochschild-Kostant-Rosenberg Theorem, Theorem 2 of \cite{Calaque_2015}]\label{Calaque-Willwacher-thm}
Under the action of the right realization, isomorphism (\ref{poly-poly-of-cdga}) becomes an isomorphism of homotopy
$\PP_2^\nu$-algebras, where the homotopy $\PP_2^\nu$-structure on the right-hand side is given by the
Tamarkin-Calaque-Willwacher map (\ref{basics-tamarkin-map}).
\end{theorem}

\subsection{Polyvectors of an operad}\label{sec4.2}

Observations made in the previous section suggest the following definition.
\begin{definition}
Given an operad $\O$ equipped with a morphism
\begin{equation*}
\comm \longrightarrow \O \:,
\end{equation*}
we define a $\omp$-algebra
\begin{equation*}
\cpol(\O) \in \alg_\omp
\end{equation*}
as the (shifted) convolution algebra
\begin{equation}\label{poly-conv-alg}
\conv(\cucoP_1\{1\}, \O)[-1] \simeq \prod_{n \ge 0} \uhom_{S_n} (\cucoP_1\{1\}(n), \O(n))[-1]
\end{equation}
twisted by a Maurer-Cartan element coming from the composite
\begin{equation*}
\PP_1^\text{nu} \longrightarrow \comm^\text{nu} \longrightarrow \comm \longrightarrow \O \:,
\end{equation*}
$\omp$-algebra structure on (\ref{poly-conv-alg}) is given by the Tamarkin-Calaque-Willwacher map (\ref{basics-tamarkin-map}).
\end{definition}
Let us note that, because the Maurer-Cartan element has only ``degree'' $2$ component, for each $k \ge 3$,
generators of the from
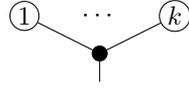
\begin{equation*}
\begin{tikzpicture}[baseline = -5]
\node[ext] (v2) at (-1,0.5) {$1$};
\node[ext] (v3) at (1,0.5) {$k$};
\node (v0) at (0,-0.5) {};
\node (vv) at (0,0.5) {$\cdots$};
\node[int] (v1) at (0,0) {};
\draw (v0) edge (v1);
\draw (v2) edge (v1);
\draw (v3) edge (v1);
\end{tikzpicture}
\end{equation*}
act on the convolution algebra (\ref{poly-conv-alg}) by zero. In particular, the commutative product on $\cpol(\O)$ is strictly 
associative. The rest of the $\omp$-structure on $\cpol(\O)$ 
consists of a strict Lie bracket and witnesses of the bracket being a homotopy bi-derivation with respect to the product, see
the description of the map (\ref{basics-tamarkin-map}).

The fact that the Lie bracket on $\cpol(\O)$ is strict will play a significant role in our constructions, so we remark that
$\omp$-action on $\cpol(\O)$ factors through the push-out
\[\xymatrix{
\lie\{1\} \ar[r] & \tP\ar@{}[dl]|(0.2)\invpb \\
\Om(\ccomm\{2\}) \ar[r]\ar[u] & \omp\ar[u] & ,
}\]
and, from now on, we will write
\begin{equation*}
\cpol(\O) \in \alg_{\tP} \:.
\end{equation*}
(Note that $\tP$ is also quasi-isomorphic to $\PP_2^\nu$. )

\begin{definition}
Note that, because $\cucoP_1 \cong \ccomm^\text{cu} \circ \colie$, we have
\begin{equation*}
\begin{split}
& \uhom_{S_n} (\cucoP_1\{1\}(n), \O(n)) \cong \\
&\cong \prod_{k \ge 0} \prod_{\stackrel{i_1 + \ldots + i_k = n}{i_1, \ldots, i_k}} \uhom_{S_k \ltimes S_{i_1} \times \ldots \times S_{i_k}} (\ccomm^\text{cu}\{1\}(k) \otimes_k \colie\{1\}(i_1) \otimes_k \ldots \otimes_k \colie\{1\}(i_k), \O(n)) \:,
\end{split}
\end{equation*}
and
\begin{equation*}
\begin{split}
& \prod_{n\ge 0} \uhom_{S_n} (\cucoP_1\{1\}(n), \O(n)) \cong \\
&\cong \prod_{k \ge 0} \prod_{i_1, \ldots, i_k} \uhom_{S_k \ltimes S_{i_1} \times \ldots \times S_{i_k}} (\ccomm^\text{cu}\{1\}(k) \otimes_k \colie\{1\}(i_1) \otimes_k \ldots \otimes_k \colie\{1\}(i_k), \O(i_1 + \ldots + i_k)) \:.
\end{split}
\end{equation*}
We define $\pol(\O) \subset \cpol(\O)$ to be the subcomplex
\begin{equation*}
\bigoplus_{k \ge 0} \prod_{i_1, \ldots, i_k} \uhom_{S_k \ltimes S_{i_1} \times \ldots \times S_{i_k}} (\ccomm^\text{cu}\{1\}(k) \otimes_k \colie\{1\}(i_1) \otimes_k \ldots \otimes_k \colie\{1\}(i_k), \O(i_1 + \ldots + i_k))
\end{equation*}
shifted by 1 to the right and equipped with an additional grading by that $k$.
\end{definition}
\begin{lemma}
The action of $\tP$ on $\cpol(\O)$ preserves the subcomplex $\pol(\O) \subset \cpol(\O)$. Moreover, the grading on $\omp$, which
we defined in Notations and Conventions at the end of Introduction, uniquely induces grading on $\tP$, and we claim that the resulting 
action of $\tP$ on $\pol(\O)$ is compatible with the gradings.
\end{lemma}

\begin{proof}
First, it follows from the explicit formula that the Lie bracket restricts to the subcomplex $\pol(\O) \subset \cpol(\O)$, and moreover
has degree $-1$ with respect to the grading.

Second, the commutative product on 
$\cpol(\O)$ is given by the $\prelie$-operation
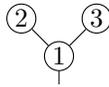
\begin{equation*}
\begin{tikzpicture}[baseline = -5]
\node (v0) at (0,-0.5) {};
\node[ext] (v1) at (0,0) {$1$};
\node[ext] (v2) at (-0.5, 0.5) {$2$};
\node[ext] (v3) at (0.5, 0.5) {$3$};
\draw (v0) edge (v1);
\draw (v1) edge (v2);
\draw (v1) edge (v3);
\end{tikzpicture}
\end{equation*}
in which we insert the composite
\begin{equation*}
\begin{split}
\cucoP_1\{1\}(2) \cong \ccomm^\text{cu}\{1\}(2) \oplus \colie\{1\}(2) &\stackrel{s^{-1}\delta*-}{\longrightarrow} \ccomm^\text{cu}\{1\}(2)[-1] \oplus \colie\{1\}(2) [-1] \longrightarrow \\ 
&\longrightarrow \colie\{1\}(2)[-1] \cong \comm(2) \subset \O(2)
\end{split}
\end{equation*}
as the first argument. The $\prelie$-operation itself has degree $-2$
with respect to the grading, while the first argument of the $\prelie$-operation has weight 2;
this results in the commutative product on $\pol(\O)$ having degree $0$.

Finally, one argues similarly for all the distributivity witnesses, which, up to the Hopf action, are given by $\prelie$-operations, 
which preserve the grading and have degrees equal to 1 minus their arity. For example, the generator
\begin{equation}\label{poly-generator-proof-grading}
X_1 \wedge \underline{X_2 X_3} \in \coP_2\{2\}(3)
\end{equation}
acts by the $\prelie$-operation
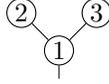
\begin{equation*}
\begin{tikzpicture}[baseline = -5]
\node (v0) at (0,-0.5) {};
\node[ext] (v1) at (0,0) {$1$};
\node[ext] (v2) at (-0.5, 0.5) {$2$};
\node[ext] (v3) at (0.5, 0.5) {$3$};
\draw (v0) edge (v1);
\draw (v1) edge (v2);
\draw (v1) edge (v3);
\end{tikzpicture}
\end{equation*}
labeled by the element $s^{-1}\underline{X_2 X_3} \in \colie\{1\}(2)[-1]$. The Hopf multiplication with the element 
$s^{-1}\underline{X_2 X_3}$ has weight 1, which, together with the fact that the $\prelie$-operation has weight -2, implies that the 
generator (\ref{poly-generator-proof-grading}) preserves the grading and has weight -1.

\end{proof}

\begin{definition}\label{poly-def-pol-gr}
We denote by
\begin{equation*}
\pol(\O) \in \alg^\gr_\tP
\end{equation*}
the resulting graded $\tP$-algebra. In addition, sending $\O \in \op_{\comm/}$ into $\pol(\O) \in \alg^\gr_\tP$ defines a functor
\begin{equation}\label{poly-fun-pol-gr}
\pol(-) \colon \op_{\comm/} \to \alg^\gr_\tP \:.
\end{equation}
\end{definition}

\begin{proposition}\label{main-prop}
The functor (\ref{poly-fun-pol-gr}) admits a left adjoint $\f \colon \alg^\gr_\tP \longrightarrow \op_{\comm/}$, and the resulting 
adjunction
\begin{equation*}
\f \colon \alg^\gr_\tP \rightleftarrows \op_{\comm/} : \pol(-)
\end{equation*}
is a Quillen adjunction.
\end{proposition}

Recall that both of the categories are equipped with projective model structures, whose equivalences and fibrations
are created by the forgetful functors to the symmetric sequences and the graded complexes correspondingly, see \cite{Hinich}.

Let us note that the functor $\pol(-)$ factors through the forgetful functor
\begin{equation*}
\op_{\comm/} \longrightarrow \op_{\comm^\nu/} \:.
\end{equation*}
We denote by $\pol^\nu(-)$ the resulting functor
\begin{equation*}
\pol^\nu(-) \colon \op_{\comm^\nu/} \longrightarrow \alg^\gr_\tP \:.
\end{equation*}

\begin{proof}[Proof of Proposition \ref{main-prop}.]
First, we note that it suffices to prove the claim for the functor $\pol^\nu(-)$. It has a left
adjoint $\f^\nu$, which admits the following description. Let $\gamma \in \alg^\gr_\tP$ be a graded $\tP$-algebra. The unit map
\begin{equation*}
\gamma \longrightarrow \pol^\nu(\f^\nu(\gamma))
\end{equation*}
sends an element $g \in \gamma$ of weight degree $w$ into a map
\begin{equation*}
\ccomm^\text{cu}\{1\}(w) \otimes_{S_w} \bigoplus_{i_1, \ldots, i_w} \mathrm{Ind}^{S_{i_1 + \ldots + i_w}}_{S_{i_1} \times \ldots \times S_{i_w}} \colie\{1\}(i_1) \otimes_k \ldots \otimes_k \colie\{1\}(i_w) \longrightarrow \bigoplus_{i_1, \ldots, i_w}\f^\nu(\gamma) (i_1 + \ldots i_w)[1] \:,
\end{equation*}
which corresponds to a family of operations
\begin{equation*}
g^c \in \f^\nu(\gamma), \quad c \in \ccomm^\text{cu}\{1\}(w) \otimes_{S_w} \bigoplus_{i_1, \ldots, i_w} \mathrm{Ind}^{S_{i_1 + \ldots + i_w}}_{S_{i_1} \times \ldots \times S_{i_w}} \colie\{1\}(i_1) \otimes_k \ldots \otimes_k \colie\{1\}(i_w) \:.
\end{equation*}
Let $\{g_i\}_i \subset \gamma$ be a family of homogeneus 
generators of algebra $\gamma$. Operad $\f^\nu(\gamma)$ could be described as the operad freely generated over $\comm^\nu$
by operations $\{g_i^c\}_{i,c}$  modulo relations on $g_i \in \gamma$ rewritten operadicaly as relations on the operations $g_i^c$.
Heuristically, for an element $g \in \gamma$ of weight degree $w$, there is an anti-symmetric homotopy multi-derivation
\begin{equation*}
g^{\Delta_w} \colon \cuccomm\{1\}(w) \longrightarrow \f^\nu (\gamma) (w)[1] \:,
\end{equation*}
whose distributivity witnesses are given all the other $g^c \in \f^\nu(\gamma)$.

Because it is clear that the functor $\pol^\nu(-)$ preserves fibrations, it remains to prove that it preserves quasi-isomorphisms. 
Because the forgetful functor
\begin{equation*}
\alg_\tP^\gr \longrightarrow \dg_k^\gr
\end{equation*}
is conservative and preserves quasi-isomorphisms, it suffices to prove that the composite
\begin{equation*}
\op_{\comm^\nu/} \longrightarrow \alg_\tP^\gr \longrightarrow \dg_k^\gr
\end{equation*}
preserves quasi-isomorphisms. Moreover, completion
\[\xymatrix{
\dg_k^\gr \ar[rr] && \dg_k \\
\oplus_i M(i) \ar@{}[u]|\invertin \ar@{|->}[rr] && \prod_i M(i) \ar@{}[u]|\invertin
}\]
is conservative and preserves quasi-isomorphisms. So it suffices to prove that the composite
\begin{equation*}
\op_{\comm^\nu/} \longrightarrow \alg_\tP^\gr \longrightarrow \dg_k^\gr \longrightarrow \dg_k
\end{equation*}
preserves quasi-isomorphisms. The latter maps an object $\mu \colon \comm^\nu \to \O$ into a complex
\begin{equation*}
\left( \prod_{n\ge0} \uhom_{S_n} (\cucoP_1\{1\}(n), \O(n))[-1], \quad d + [\mu, -] \right) \:.
\end{equation*}
Let us note that the MC-element $\mu$ is given by
\begin{equation*}
\begin{split}
\cucoP_1\{1\}(2) &\twoheadrightarrow \colie\{1\}(2) \cong \colie(2) \otimes_k \sgn_{S_2}[1] \cong \\
&\cong \triv_{S_2}[1] \cong \comm(2)[1] \stackrel{\mu}{\longrightarrow} \O(2)[1] \:;
\end{split}
\end{equation*}
in particular, it preserves the differential, and the Maurer-Cartan equation simplifies to
\begin{equation*}
[\mu, \mu] = 0 \:.
\end{equation*}
From this we conclude that the functor
\begin{equation*}
\op_{\comm^\mathrm{nu}/} \to \dg_k
\end{equation*}
factors over the realization functor
\begin{equation*}
|-| \colon \dg^{\text{gr,}\epsilon}_k \longrightarrow \dg_k \:;
\end{equation*}
the object $\mu \colon \comm^\nu \to \O$ is mapped into a graded complex $\oplus_{n \ge 0} \uhom_{S_n} (\cucoP_1\{1\}(n), \O(n))$ 
equipped with mixed
differential $[\mu, -]$. It remains to recall that the realization preserves quasi-isomorphisms, see \cite{CPTVV}, Section 1.3, or
\cite{Melani_2018}, Section 1.2.

\end{proof}

\begin{notation}
The adjunction from Proposition \ref{main-prop} induces an adjunction
\begin{equation*}
\F \colon \Alg^\gr_{\PP^\nu_2} \simeq \alg^\gr_{\tP}[W^{-1}_\mathrm{qis}] \rightleftarrows \op_{\comm/}[W^{-1}_\mathrm{qis}] : \Pol(-)
\end{equation*}
between $\infty$-categories. Also, there is a ``non-unital'' variant of this adjunction, which we write with a superscript $\nu$:
\begin{equation*}
\F^\nu \colon \Alg^\gr_{\PP^\nu_2} \rightleftarrows \op_{\comm^\nu/}[W^{-1}_\mathrm{qis}] : \Pol^\nu \:.
\end{equation*}
\end{notation}

\subsection{Degenerate Poisson operads}

It remains to make a few final remarks before we can define the degenerate Poisson operads. Let
\begin{equation*}
\Free^{\PP_2^\nu}_{\lie\{1\}} (k(2)[-2]) \in \Alg^\gr_{\PP_2^\nu}
\end{equation*}
denote the free graded
$\PP^\text{nu}_2$-algebra generated by a trivial one-dimensional graded Lie algebra $k(2)[-1]$ sitting in weight degree 2 and 
cohomological degree 1.
\begin{proposition}\label{poisson-operad}
There is an equivalence of operads
\begin{equation*}
\PP_1 \simeq \F(\Free^{\PP_2^\nu}_{\lie\{1\}} (k(2)[-2])) \:,
\end{equation*}
compatible with the structure maps from $\comm$.
\end{proposition}
The following lemma will help us to reduce the question to the ``non-unital'' one.
\begin{lemma}\label{push-out-square}
The square
\[\xymatrix{
\comm \ar[r] & \PP_1 \\
\comm^{\text{nu}} \ar[r]\ar[u] & \PP_1^{\text{nu}} \ar[u]
}\]
is quasi-isomorphic to a homotopy push-out square.
\end{lemma}
\begin{proof}[Proof of the lemma.]

There is a diagram
\[\xymatrix{
\Omega(\ccolie\{1\}) && \\
\Omega(\colie\{1\}) \ar[r]\ar[u] & \Omega(\coP_1\{1\}) & \quad,
}\]
whose maps are the obvious ones defined on the generators,
that resolves the diagram
\[\xymatrix{
\comm && \\
\comm^{\text{nu}} \ar[r]\ar[u] & \PP^{\text{nu}}_1 & .
}\]
Let us denote push-out of the former diagram by $\P$. Operad $\P$ is a quasi-free operad generated by the symmetric sequence
\begin{equation*}
\overline{\ccolie\{1\}}[-1] \oplus ( \overline{\ccomm\{1\}} \circ \colie\{1\}[-1] )
\end{equation*}
and equipped with a cobar-type differential. In particular, there is a split injection
\begin{equation}\label{map-between-poisson-resolutions}
\P \longrightarrow \Omega(\ccoP_1\{1\}) \: :
\end{equation}
on the generators the splitting is given by the obvious map
$\overline{\ccoP_1\{1\}}[-1] \longrightarrow \P$ and it is straightforward to check that the resulting map of operads is compatible 
with the differential. In addition, one could also check that the composite
\begin{equation*}
\P \longrightarrow \Omega(\ccoP_1\{1\}) \stackrel{\simeq}{\longrightarrow} \PP_1
\end{equation*}
is surjective, which amounts to checking that $\P$ has enough generators. We conclude that the morphism
(\ref{map-between-poisson-resolutions}) is a quasi-isomorphism. Finally, we note that
the push-out square
\[\xymatrix{
\Omega(\ccolie\{1\}) \ar[r] & \P \\
\Omega(\colie\{1\}) \ar[r]\ar[u] & \Omega(\coP_1\{1\}) \ar[u]
}\]
is in fact a homotopy push-out square -- all the objects are cofibrant and all the maps are cofibrations. The latter follows from
an explicit description of cofibrations in the model category of operads, see \cite{Hinich}.
This concludes the proof of Lemma
\ref{push-out-square}.
\end{proof}

\begin{proof}[Proof of Proposition \ref{poisson-operad}.]
First, let us prove that
$\F^\text{nu}(\Free^{\PP_2^\nu}_{\lie\{1\}} (k(2)[-2])) \simeq \PP^\text{nu}_1$. 
Let us denote by $\f^\nu_\lie$ the left adjoint to the composite
\begin{equation*}
\op_{\comm^\nu/} \stackrel{\pol^\nu}{\longrightarrow} \alg^\gr_\tP \stackrel{\oblv^{\lie\{1\}}_\tP}{\longrightarrow} \alg^\gr_{\lie\{1\}} \stackrel{(-)[1]}{\longrightarrow} \alg^\gr_\lie \:.
\end{equation*}
The functor $\f^\nu_\lie$ is a left Quillen functor which maps a graded Lie algebra into the operad generated over $\comm^\nu$ by
homotopy multi-derivations corresponding to elements of the Lie algebra modulo relations that come from the Lie bracket and 
the differential of the Lie algebra.

Recall from \cite{Melani}, Section 4, that Lie algebra $k(2)[-1]$ has a cofibrant resolution
$Q(k(2)[-1])$ which is freely generated by elements $p_i$ for $i = 2, 3, 4, \ldots$, such that $p_i$ has cohomological degree $1$ 
and weight degree $i$; the differential acts on the generators by the formula
\begin{equation*}
d p_n = \sum_{i + j = n+1} \frac{1}{2}[p_i, p_j] \:.
\end{equation*}
The map
\begin{equation*}
Q(k(2)[-1]) \longrightarrow k(2)[-1]
\end{equation*}
sends $p_2$ to the generator of $k(2)[-1]$ and the other generators to zero.

We get an equivalence
\begin{equation*}
\F^\nu(\Free^{\PP_2^\nu}_{\lie\{1\}} (k(2)[-2])) \simeq \f_\lie^\nu (Q(k(2)[-1])) \:,
\end{equation*}
and the latter is freely generated over $\comm^\nu$ by the following set of operations:
for each $n \ge 2$, there is a family of operations $p_n^c$ labeled by cooperations
\begin{equation*}
c \in \ccomm\{1\}(w) \otimes_{S_w} \bigoplus_{i_1, \ldots, i_w} \mathrm{Ind}^{S_{i_1 + \ldots + i_w}}_{S_{i_1} \times \ldots \times S_{i_w}} \colie\{1\}(i_1) \otimes_k \ldots \otimes_k \colie\{1\}(i_w) \:;
\end{equation*}
arity of the operation $p_n^c$ is equal to arity of the cooperation $c$, the symmetric group acts on $\{p_n^c\}_c$ via acting on the 
labels.
In particular, we see that the underlying graded operad of $\f^\nu_\lie(Q(k(2)[-1]))$ is isomorphic to the underlying graded operad
of the push-out of the diagram
\begin{equation}\label{poly-diag-1}
\xymatrix{
\comm^\nu & \\
\Omega(\colie\{1\}) \ar[u]\ar[r] & \Omega(\coP_1\{1\})& .
}
\end{equation}
It is a straightforward computation that this isomorphism is compatible with the differential; to give an idea of how the computation
is performed, in Example \ref{computation1}, we compute action of the differential of $F^\nu_\lie(Q(k(2)[-1]))$ on a few simple generators.
On the other hand, the push-out of the 
diagram (\ref{poly-diag-1})
is a cofibrant over $\comm^\nu$ model of $\PP^\nu_1$; let us denote this model $\overline{P}{}^\nu$. So far, we proved that
\begin{equation*}
\F^\text{nu}(\Free^{\PP_2^\nu}_{\lie\{1\}} (k(2)[-2])) \simeq F^\nu_{\lie}(Q(k(2)[-1])) \cong \overline{P}{}^\nu \simeq \PP^\nu_1 \:.
\end{equation*}

Now let us denote by $\overline{P}$ the push-out of the diagram
\[\xymatrix{
\comm & \\
\comm^\nu \ar[r]\ar[u] & \overline{P}{}^\nu &.
}\]
It needs to be proved that $\overline{P}$ is quasi-isomorphic to $\PP_1$. But it follows from the definitions and the proof of 
Lemma \ref{push-out-square} that the operad $\overline{P}$ fits into push-out square
\[\xymatrix{
\comm \ar[r] & \overline{P} \ar@{}[dl]|(0.25)\invpb \\
\Omega(\colie^\theta\{1\}) \ar[r]\ar[u] & \P \ar[u] & .
}\]
\end{proof}

\begin{example}\label{computation1}
To simplify the notation, in this example, we would like to denote the operad $F^\nu_\lie(Q(k(2)[-1]))$ by $\O$. Here our goal will
be to compute how the differential $d_\O$ acts on a few simple generators. Recall that $\O$ is quasi-freely generated over $\comm^\nu$ 
by the symmetric sequence
\begin{equation*}
\left( \coP_1\{1\} / \colie\{1\} \right) [1] \:.
\end{equation*}
The unit map
\begin{equation*}
Q(k(2)[-1]) \longrightarrow \pol^\nu(\O)
\end{equation*}
maps a generator $p_n \in Q(k(2)[-1])$ into a map
\begin{equation*}
\ccomm\{1\}(n) \otimes_{S_n} \bigoplus_{i_1, \ldots, i_n} \mathrm{Ind}^{S_{i_1 + \ldots + i_n}}_{S_{i_1} \times \ldots \times S_{i_n}} \colie\{1\}(i_1) \otimes_k \ldots \otimes_k \colie\{1\}(i_n) \longrightarrow \bigoplus_{i_1, \ldots, i_n}\O(i_1 + \ldots i_n)[1] \:,
\end{equation*}
which we also denote by $p_n$; this map sends a cooperation into the corresponding generator of $\O$. Also, recall that we have a map
\begin{equation*}
\mu \colon \colie\{1\}(2) \longrightarrow \O(2)[1]
\end{equation*}
that corresponds to the commutative product $\comm^\nu(2) \subset \O(2)$; the differential on $\pol^\nu(\O)$ involves taking the bracket
with $\mu$.
Let
\begin{equation*}
\Delta_2, \delta \in \coP_1\{1\}(2)
\end{equation*}
denote the cocommutative cobracket and the Lie cobracket correspondingly. Also, let $\Delta_3 \in \coP_1\{1\}(3)$ denote the cooperation
\begin{equation*}
\Delta_3 \stackrel{\define}{=} \Delta_2 \circ_1 \Delta_2 \in \ccomm\{1\}(3) \subset \coP_1\{1\}(3) \:.
\end{equation*}
\begin{enumerate}
\item [1)] $p_2(\Delta_2) \in \O(2)$ is an anti-symmetric binary operation of cohomological degree 0.
\begin{equation*}
d_\O (p_2(\Delta_2)) = (dp_2 + [\mu, p_2]) (\Delta_2) = [\mu, p_2] (\Delta_2) = 0 \:.
\end{equation*}
\item[2)] $p_3(\Delta_3) \in \O(3)$ is an anti-symmetric operation of cohomological degree -1.
\begin{equation*}
\begin{split}
d_\O(p_3(\Delta_3)) &= (dp_3 + [\mu, p_3]) (\Delta_3) = (\frac{1}{2}[p_2, p_2] + [\mu, p_3]) (\Delta_3) = \\ 
&= \frac{1}{2}[p_2, p_2] (\Delta_3) = \gamma_\O \circ (p_2 \otimes p_2) \circ \overline{\Delta}_{(1)} (\Delta_3) = \\
&= \gamma_\O \circ (p_2 \otimes p_2) ( \Delta_2 \circ_1 \Delta_2 + (\Delta_2 \circ_1 \Delta_2)^{(2 3 1)} + (\Delta_2 \circ_1 \Delta_2)^{(3 1 2)} ) = \\
&= p_2(\Delta_2) \circ_1 p_2(\Delta_2) + (p_2(\Delta_2) \circ_1 p_2(\Delta_2))^{(231)} + (p_2(\Delta_2) \circ_1 p_2(\Delta_2))^{(312)} \:.
\end{split}
\end{equation*}
\item[3)] $p_2 \left(\begin{tikzpicture}[baseline = -5]
\node (v1) at (0.5,0.5) {};
\node[ext] (v2) at (-0.5,0.5) {$\delta$};
\node[ext] (v3) at (0, 0) {$\Delta_2$};
\node (vr) at (0, -0.7) {};
\draw (v1) edge (v3);
\draw (v2) edge (v3);
\draw (v3) edge (vr);
\node (a1) at (-0.7, 1) {};
\node (a2) at (-0.3, 1) {};
\node (a3) at (0.5, 1) {};
\draw (a1) edge (v2);
\draw (a2) edge (v2);
\end{tikzpicture}\right) \in \O(3)$ is an operation of cohomological degree $-1$.
\begin{equation*}
\begin{split}
d_\O (p_2 \left(\begin{tikzpicture}[baseline = -5]
\node (v1) at (0.5,0.5) {};
\node[ext] (v2) at (-0.5,0.5) {$\delta$};
\node[ext] (v3) at (0, 0) {$\Delta_2$};
\node (vr) at (0, -0.7) {};
\draw (v1) edge (v3);
\draw (v2) edge (v3);
\draw (v3) edge (vr);
\node (a1) at (-0.7, 1) {};
\node (a2) at (-0.3, 1) {};
\node (a3) at (0.5, 1) {};
\draw (a1) edge (v2);
\draw (a2) edge (v2);
\end{tikzpicture}\right) ) &= (dp_2 + [\mu, p_2]) \left(\begin{tikzpicture}[baseline = -5]
\node (v1) at (0.5,0.5) {};
\node[ext] (v2) at (-0.5,0.5) {$\delta$};
\node[ext] (v3) at (0, 0) {$\Delta_2$};
\node (vr) at (0, -0.7) {};
\draw (v1) edge (v3);
\draw (v2) edge (v3);
\draw (v3) edge (vr);
\node (a1) at (-0.7, 1) {};
\node (a2) at (-0.3, 1) {};
\node (a3) at (0.5, 1) {};
\draw (a1) edge (v2);
\draw (a2) edge (v2);
\end{tikzpicture}\right) = [\mu, p_2] \left(\begin{tikzpicture}[baseline = -5]
\node (v1) at (0.5,0.5) {};
\node[ext] (v2) at (-0.5,0.5) {$\delta$};
\node[ext] (v3) at (0, 0) {$\Delta_2$};
\node (vr) at (0, -0.7) {};
\draw (v1) edge (v3);
\draw (v2) edge (v3);
\draw (v3) edge (vr);
\node (a1) at (-0.7, 1) {};
\node (a2) at (-0.3, 1) {};
\node (a3) at (0.5, 1) {};
\draw (a1) edge (v2);
\draw (a2) edge (v2);
\end{tikzpicture}\right) = \\
&= \gamma_\O \circ (\mu \otimes p_2 + p_2 \otimes \mu) \circ \overline{\Delta}_{(1)} \left(\begin{tikzpicture}[baseline = -5]
\node (v1) at (0.5,0.5) {};
\node[ext] (v2) at (-0.5,0.5) {$\delta$};
\node[ext] (v3) at (0, 0) {$\Delta_2$};
\node (vr) at (0, -0.7) {};
\draw (v1) edge (v3);
\draw (v2) edge (v3);
\draw (v3) edge (vr);
\node (a1) at (-0.7, 1) {};
\node (a2) at (-0.3, 1) {};
\node (a3) at (0.5, 1) {};
\draw (a1) edge (v2);
\draw (a2) edge (v2);
\end{tikzpicture}\right) =\\
&= \gamma_\O \circ (\mu \otimes p_2 + p_2 \otimes \mu) ( \Delta_2 \circ_1 \delta + (\delta \circ_1 \Delta_2)^{(231)} + (\delta \circ_1 \Delta_2)^{(312)} ) =\\
&= p_2(\Delta_2) \circ_1 \mu(\delta) + ( \mu(\delta) \circ_1 p_2(\Delta_2) )^{(231)} + ( \mu(\delta) \circ_1 p_2(\Delta_2) )^{(312)} \:.
\end{split}
\end{equation*}
\end{enumerate}
\end{example}

\begin{notation}
Let 
\begin{equation*}
C \stackrel{\define}{=} \sym^{\ge 1}_k(k(2)[-2]) \in \alg^\gr_{\comm^\nu}
\end{equation*} denote the free non-unital graded commutative algebra generated by one element of weight degree $2$ and
cohomological degree $2$. Let us denote the generator by $x$. We will denote by $C_{\le m}$ the quotient
\begin{equation*}
C / (x^{m+1})
\end{equation*}
of $C$ by the ideal $(x^{m+1}) \subset C$. 
\end{notation}

\begin{definition}\label{main-def}
Recall that there is a functor
\begin{equation*}
\Triv^{\PP^\nu_2}_{\comm^\nu} = \Oblv^{\PP^\nu_2}_{\comm^\nu} \colon \Alg^\gr_{\comm^\nu} \longrightarrow \Alg^\gr_{\PP^\nu_2}
\end{equation*}
that is the right adjoint corresponding to the map of operads
\begin{equation*}
\PP^\nu_2 \longrightarrow \comm^\nu \:.
\end{equation*}
This functor equips a commutative algebra with a trivial bracket.
We define the $m$-degenerate Poisson operad $\PP_1^{\le m} \in \op_{\comm/}$ by the formula
\begin{equation*}
\PP_1^{\le m} = \F( \Triv^{\PP^\nu_2}_{\comm^\nu} (C_{\le m}) ) \:,
\end{equation*}
The operad $\PP_1^{\le m}$ receives a map from
\begin{equation*}
\begin{split}
\PP_1 &\simeq \F(\Free^{\PP^\nu_2}_{\lie\{1\}} \circ \Triv^{\lie\{1\}} (k(2)[-2]) ) \simeq \F( \Triv^{\PP_2^\nu}_{\comm^\nu} \circ \Free^{\comm^\nu} (k(2)[-2]) ) \\
&\simeq \F( \Triv^{\PP^\nu_2}_{\comm^\nu}(C)) \:.
\end{split}
\end{equation*}
\end{definition}

\begin{remark}
By definition, the $m$-degenerate Poisson operad satisfies a universal property
\begin{equation*}
\map_{\op_{\comm/}} (\PP_1^{\le m}, \O) \simeq \map_{\Alg^\gr_{\PP^\nu_2}} ( \Triv^{\PP^\nu_2}_{\comm^\nu} (C_{\le m}) , \Pol(\O)) \:,
\end{equation*}
which, in particular, similarly to \cite{Melani}, Theorem 3.2, computes us the space of $m$-degenerate Poisson structures on a 
(differential-graded) commutative algebra in terms of the polyvectors on the latter.
\end{remark}

\subsection{A simplicial resolution of $\PP^{\le m}_1$}

In this section we would like to describe a simplicial resolution of the operad $\PP_1^{\le m}$. Let $k[x] \in \alg_\comm^\gr$
denote the unital polynomial ring in $1$ variable $x$ of cohomological degree $2$ and weight degree $2$. Let us also denote by
$k[y] \subset k[x]$ its subring given by $y \mapsto x^{m+1}$. Recall that there is a simiplicial resolution
\begin{equation*}
k[x]/x^m \cong k[x] \otimes^\mathrm{L}_{k[y]} k \stackrel{\cong}{\longleftarrow} | k[x] \quad\substack{\longleftarrow \\ \longleftarrow}\quad k[x]\otimes_k k[y] \quad\substack{\longleftarrow \\ \longleftarrow \\ \longleftarrow}\quad k[x]\otimes_k k[y]\otimes_k k[y] \quad\substack{\longleftarrow \\ \longleftarrow \\ \longleftarrow \\ \longleftarrow}\quad \ldots |
\end{equation*}
of the factor algebra $k[x] / x^{m+1}$. Moreover, units of the terms of this resolution assemble to a split map
\[\xymatrix{
k\ar[d] & | k\ar[d] \ar[l]_\cong & k\ar[d] \ar@<0.5ex>[l] \ar@<-0.5ex>[l] & k\ar[d] \ar[l]\ar@<1ex>[l]\ar@<-1ex>[l] & \ldots &| \\
k[x]/(x^{m+1}) & | k[x] \ar[l]_(0.4)\cong & k[x]\otimes_k k[y] \ar@<0.5ex>[l]\ar@<-0.5ex>[l] & k[x]\otimes_kk[y]\otimes_kk[y] \ar[l]\ar@<1ex>[l]\ar@<-1ex>[l] & \ldots &|
}\]
of simplicial objects. Cokernel of this map gives us a simplicial resolution
\[\xymatrix{
C_{\le m} & | C_m^{(0)} \ar[l]_\cong & C^{(1)}_m \ar@<0.5ex>[l]\ar@<-0.5ex>[l] & C^{(2)}_m \ar[l]\ar@<1ex>[l]\ar@<-1ex>[l] & \ldots &| &,
}\]
where $C_m^{(i)}$ denotes free non-unital commutative algebra generated by weight graded complex
\begin{equation}\label{model-generating-complex}
k(2)[-2] \oplus \underbrace{k(2m+2)[-2m-2] \oplus \ldots \oplus k(2m+2)[-2m-2]}_i \quad.
\end{equation}
Because functors $\F$ and $\Triv^{\PP^\nu_2}_{\comm^\nu}$ preserve geometric realizations, the simplicial resolution of $C_{\le m}$
gives rise to a simplicial resolution
\[\xymatrix{
\PP_1^{\le m} & | \F ( \Triv^{\PP^\nu_2}_{\comm^\nu} ( C_m^{(0)} )) \ar[l]_(0.7)\cong & \F ( \Triv^{\PP^\nu_2}_{\comm^\nu} ( C^{(1)}_m )) \ar@<0.5ex>[l]\ar@<-0.5ex>[l] & \ar[l]\ar@<1ex>[l]\ar@<-1ex>[l] & \ldots &|
}\]
of the $m$-degenerate Poisson operad. Let us denote by
\begin{equation*}
L_m^{(i)} \stackrel{\define}{=} \triv^{\lie\{1\}} (k(2)[-2] \oplus \underbrace{k(2m+2)[-2m-2] \oplus \ldots \oplus k(2m+2)[-2m-2]}_i) \in \alg^\gr_{\lie\{1\}}
\end{equation*}
weight graded complex (\ref{model-generating-complex}) equipped with a trivial $1$-shifted Lie bracket. We note that the
equivalence
\begin{equation*}
\Free^{\PP_2^\nu}_{\lie\{1\}} \circ \Triv^{\lie\{1\}} \simeq \Triv^{\PP_2^\nu}_{\comm^\nu} \circ \Free^{\comm^\nu} \:,
\end{equation*}
which produces an equivalence of algebras
\begin{equation*}
\Free^{\PP_2^\nu}_{\lie\{1\}} (L_m^{(i)}) \simeq \Triv^{\PP_2^\nu}_{\comm^\nu} ( C_m^{(i)} ) \:,
\end{equation*}
allows us to rewrite the resolution of $\PP_1^{\le m}$ as
\[\xymatrix{
\PP_1^{\le m} & | \F_\lie ( L_m^{(0)}[1] ) \ar[l]_(0.6)\cong & \F_\lie  ( L^{(1)}_m[1] ) \ar@<0.5ex>[l]\ar@<-0.5ex>[l] & \ar[l]\ar@<1ex>[l]\ar@<-1ex>[l] & \ldots &| &.
}\]
\begin{remark}\label{model-term-operad}
It is possible to describe the terms of the resolution explicitly: $\lie\{1\}$-algebra $L_m^{(i)}$ has a bar-cobar resolution
$\widetilde{L_m^{(i)}}$, which is a semi-free $\lie\{1\}$-algebra generated by
\begin{equation*}
\ccomm(L_m^{(i)}) = \bigoplus_{n \ge 1} \sym_k^n (L_m^{(i)})
\end{equation*}
and equipped with a differential that acts on the generators by the formula
\begin{equation}\label{model-term-differential}
\sym_k^{n} (L_m^{(i)}) \stackrel{\Delta}{\longrightarrow} \bigoplus_{\substack{p + q = n \\ p,q \ge 1}} \sym_k^p(L_m^{(i)}) \otimes_k \sym_k^q(L_m^{(i)}) \stackrel{[-,-]}{\longrightarrow} \lie\{1\}(\ccomm(L_m^{(i)})) \:.
\end{equation}
Moreover, the weight grading on $L_m^{(i)}$ uniquely extends to $\widetilde{L_m^{(i)}}$ under the condition that the differential
(\ref{model-term-differential}) preserves the weight grading and the Lie bracket has weight degree $-1$. So we obtain that
\begin{equation*}
\F_\lie(L_m^{(i)}[1]) \simeq \f_\lie (\widetilde{L_m^{(i)}}[1]) \:,
\end{equation*}
and the latter has an ``explicit'' description: it is the operad semi-freely generated over $\comm$
by the anti-symmetric homotopy multi-derivations that correspond to the elements of
\begin{equation*}
\bigoplus_{n \ge 1} \sym_k^n (L_m^{(i)})
\end{equation*}
together with their homotopy coherences; the operad is equipped with a differential that is induced by
(\ref{model-term-differential}).
\end{remark}
Let us note that it follows from this description that each of the terms of the resolution sits in non-positive cohomological
degrees, which implies that $\PP_1^{\le m}$ itself sits in non-positive cohomological degrees. In addition, the resolution allows us
to compute the operad $H^0(\PP_1^{\le m})$.
\begin{proposition}\label{model-classical-operad}
Operad $H^0(\PP_1^{\le m})$ is isomorphic to factor of $\PP_1$ by an operadic ideal generated by the anti-symmetrization of the
operation
\begin{equation*}
(\ldots (\mu_{m+1} \circ_{m+1} l) \ldots \circ_1 l) \in \PP_1(2m+2) \:,
\end{equation*}
where $\mu_{m+1} \in \PP_1(m+1)$ denotes the $m+1$-arity commutative product and $l \in \PP_1(2)$ denotes the Lie bracket.
\end{proposition}

\begin{proof}
First, let us note that, by the virtue of Proposition \ref{poisson-operad}, there is a quasi-isomorphism
\begin{equation*}
\f_\lie(\widetilde{L_m^{(0)}}) \longrightarrow \PP_1 \:.
\end{equation*}
Therefore, $H^0(\PP_1^{\le m})$ is isomorphic to factor of $\PP_1$ by the image of the composite
\begin{equation*}
\f_\lie(\widetilde{L_m^{(1)}}) \stackrel{d}{\longrightarrow} \f_\lie(\widetilde{L_m^{(0)}}) \longrightarrow \PP_1 \:.
\end{equation*}
That image is an operadic ideal generated by the operation of $\PP_1$ that corresponds to element $x^{m+1} \in C^{(0)}_m \subset k[x]$.
Noting that $x$ itself corresponds to the Lie bracket $l \in \PP_1(2)$ and unwinding the definition of the 
product on $\pol(\PP_1)$, we get the claimed description of the ideal; in Example \ref{xsquared}, we will perform this calculation in 
the case where $m = 1$.
\end{proof}

\begin{example}\label{xsquared}
In this example, we will present a special case of the calculation we skipped in the proof of Proposition \ref{model-classical-operad}.
For degree reasons, the composite
\begin{equation*}
\widetilde{L_m^{(0)}} \longrightarrow \pol(\f_\lie(\widetilde{L_m^{(0)}}))[1] \longrightarrow \pol(\PP_1)[1]
\end{equation*}
factors through the projection $\widetilde{L_m^{(0)}} \longrightarrow k(2)[-1]$. The generator of $k(2)[-1]$ is sent into
a map
\begin{equation*}
\ccomm\{1\}(2) \otimes_{S_2} \bigoplus_{i_1, i_2} \mathrm{Ind}^{S_{i_1 + i_2}}_{S_{i_1} \times S_{i_2}} \colie\{1\}(i_1) \otimes_k \colie\{1\}(i_2) \longrightarrow \bigoplus_{i_1, i_2} \PP_1(i_1 + i_2)[1]
\end{equation*}
that maps $\Delta_2 \in \ccomm\{1\}(2)$ into the Lie bracket $l \in \PP_1(2)$ and everything else to zero.
In agreement with the notation from the proof of Proposition \ref{model-classical-operad}, we denote this map by $x$.
Here we would like to compute $x^2 \in \pol(\PP_1)$. By the definitions from Section \ref{sec2},
\begin{equation*}
x^2 = \:\begin{tikzpicture}[baseline = -5]
\node[ext] (v2) at (-1,0.5) {$1$};
\node[ext] (v3) at (1,0.5) {$2$};
\node (v0) at (0,-0.5) {};
\node[int] (v1) at (0,0) {};
\draw (v0) edge (v1);
\draw (v2) edge (v1);
\draw (v3) edge (v1);
\end{tikzpicture} (x, x)  = \:\begin{tikzpicture}[baseline = -5]
\node[ext] (v2) at (-1,0.5) {$2$};
\node[ext] (v3) at (1,0.5) {$3$};
\node (v0) at (0,-0.5) {};
\node[ext] (v1) at (0,0) {$1$};
\draw (v0) edge (v1);
\draw (v2) edge (v1);
\draw (v3) edge (v1);
\end{tikzpicture} (\widetilde{\mu}, x, x) \:,
\end{equation*}
where $\widetilde{\mu}$ is the composite
\begin{equation*}
\ccomm\{1\}(2) \stackrel{(s^{-1}\delta) * -}{\longrightarrow} \colie\{1\}(2)[-1] \cong \comm(2) \hookrightarrow \PP_1(2) \:,
\end{equation*}
it sends $\Delta_2$ to the commutative product $\mu_2 \in \PP_1(2)$. (Note that this map is not a map of $S_2$-representations and it
does not have to be. ) By the composition law of the preLie operad,
\begin{equation*}
\begin{split}
\begin{tikzpicture}[baseline = -5]
\node[ext] (v2) at (-1,0.5) {$2$};
\node[ext] (v3) at (1,0.5) {$3$};
\node (v0) at (0,-0.5) {};
\node[ext] (v1) at (0,0) {$1$};
\draw (v0) edge (v1);
\draw (v2) edge (v1);
\draw (v3) edge (v1);
\end{tikzpicture} (\widetilde{\mu}, x, x) &= \: (\begin{tikzpicture}[baseline = -5]
\node[ext] (v1) at (0,-0.5) {$1$};
\node[ext] (v2) at (0,0.5) {$2$};
\node (v3) at (0, -1) {};
\draw (v1) edge (v2);
\draw (v1) edge (v3);
\end{tikzpicture}
\: \circ_1 \:
\begin{tikzpicture}[baseline = -5]
\node[ext] (v1) at (0,-0.5) {$1$};
\node[ext] (v2) at (0,0.5) {$2$};
\node (v3) at (0, -1) {};
\draw (v1) edge (v2);
\draw (v1) edge (v3);
\end{tikzpicture} \:-\: \begin{tikzpicture}[baseline = -5]
\node[ext] (v1) at (0,-0.5) {$1$};
\node[ext] (v2) at (0,0.5) {$2$};
\node (v3) at (0, -1) {};
\draw (v1) edge (v2);
\draw (v1) edge (v3);
\end{tikzpicture}
\: \circ_2 \:
\begin{tikzpicture}[baseline = -5]
\node[ext] (v1) at (0,-0.5) {$1$};
\node[ext] (v2) at (0,0.5) {$2$};
\node (v3) at (0, -1) {};
\draw (v1) edge (v2);
\draw (v1) edge (v3);
\end{tikzpicture}) (\widetilde{\mu}, x, x) = \\
&= (\widetilde{\mu} \star x) \star x - \widetilde{\mu} \star (x \star x) \:.
\end{split}
\end{equation*}
It follows that $x^2 \in \pol(\PP_1)$ maps everything to zero except $\Delta_4 \in \ccomm\{1\}(4)$.
Let us now compute $x^2(\Delta_4)$:
\begin{equation*}
\begin{split}
((\widetilde{\mu} \star x) \star x) (\Delta_4) &= \sum_{\s \in \sh(2,2)} (-1)^\s ( (\widetilde{\mu} \star x) (\Delta_3) \circ_1 l )^{\s^{-1}} = \\
&= \sum_{\s \in \sh(2,2)} (-1)^\s ( (\mu_2 \circ l) \circ_1 l + (\mu_2 \circ_1 l)^{(231)} \circ_1 l + (\mu_2 \circ_1 l)^{(312)} \circ_1 l )^{\s^{-1}} = \\
&= \sum_{\s\in\sh(2,2)} (-1)^\s ((\mu_2 \circ_1 l)^{(231)} \circ_1 l)^{\s^{-1}} + R \:; \\
R\quad &= \sum_{\s \in \sh(2,2)} (-1)^\s ( (\mu_2 \circ l) \circ_1 l + (\mu_2 \circ_1 l)^{(312)} \circ_1 l )^{\s^{-1}} =\\
&= \sum_{\s \in \sh(2,2)} (-1)^\s ( (\mu_2 \circ l) \circ_1 l - ((\mu_2 \circ_1 l) \circ_1 l)^{(1243)} )^{\s^{-1}} = \\
=((\mu_2 \circ_1 l) \circ_1 l) - ((\mu_2 \circ_1 l) \circ_1 l)^{(1243)} &- ((\mu_2 \circ_1 l) \circ_1 l)^{(1324)} + ((\mu_2 \circ_1 l) \circ_1 l)^{(1342)} + \\
+ ((\mu_2 \circ_1 l) \circ_1 l)^{(1423)} - ((\mu_2 \circ_1 l) \circ_1 l)^{(1432)} &- ((\mu_2 \circ_1 l) \circ_1 l)^{(2413)} + ((\mu_2 \circ_1 l) \circ_1 l)^{(2431)} + \\
+ ((\mu_2 \circ_1 l) \circ_1 l)^{(3412)} - ((\mu_2 \circ_1 l) \circ_1 l)^{(3421)} &+ ((\mu_2 \circ_1 l) \circ_1 l)^{(2314)} - ((\mu_2 \circ_1 l) \circ_1 l)^{(2341)} = \\
= ((\mu_2 \circ_1 l) \circ_1 l) &+ ((\mu_2 \circ_1 l) \circ_1 l)^{(2341)} + ((\mu_2 \circ_1 l) \circ_1 l)^{(3124)} - \\
- ((\mu_2 \circ_1 l) \circ_1 l)^{(2341)} &- ((\mu_2 \circ_1 l) \circ_1 l)^{(3421)} - ((\mu_2 \circ_1 l) \circ_1 l)^{(4231)} + \\
+ ((\mu_2 \circ_1 l) \circ_1 l)^{(3412)} &+ ((\mu_2 \circ_1 l) \circ_1 l)^{(4132)} + ((\mu_2 \circ_1 l) \circ_1 l)^{(1342)} - \\
- ((\mu_2 \circ_1 l) \circ_1 l)^{(4123)} &- ((\mu_2 \circ_1 l) \circ_1 l)^{(1243)} - ((\mu_2 \circ_1 l) \circ_1 l)^{(2413)} = \\
= ( (\mu_2 \circ_1 (l \circ_1 l)) &+ (\mu_2 \circ_1 (l \circ_1 l)^{(231)}) + (\mu_2 \circ_1 (l \circ_1 l)^{(312)}) ) - \\
- ( (\mu_2 \circ_1 (l \circ_1 l)) &+ (\mu_2 \circ_1 (l \circ_1 l)^{(231)}) + (\mu_2 \circ_1 (l \circ_1 l)^{(312)}) )^{(2341)} + \\
+ ( (\mu_2 \circ_1 (l \circ_1 l)) &+ (\mu_2 \circ_1 (l \circ_1 l)^{(231)}) + (\mu_2 \circ_1 (l \circ_1 l)^{(312)}) )^{(3412)} - \\
- ( (\mu_2 \circ_1 (l \circ_1 l)) &+ (\mu_2 \circ_1 (l \circ_1 l)^{(231)}) + (\mu_2 \circ_1 (l \circ_1 l)^{(312)}) )^{(4123)} = \\
= (\mu_2 \circ_1 ((x \star x)(\Delta_3))) &- (\mu_2 \circ_1 ((x \star x)(\Delta_3)))^{(2341)} + \\
+ (\mu_2 \circ_1 ((x \star x)(\Delta_3)))^{(3412)} &- (\mu_2 \circ_1 ((x \star x)(\Delta_3)))^{(4123)} = \\
&= (\widetilde{\mu} \star (x\star x)) (\Delta_4) \:.
\end{split}
\end{equation*}
This computation shows that
\begin{equation*}
x^2(\Delta_4) = \sum_{\s\in\sh(2,2)} (-1)^\s ((\mu_2 \circ_1 l)^{(231)} \circ_1 l)^{\s^{-1}} = \sum_{\s\in\sh(2,2)} (-1)^\s ((\mu_2 \circ_2 l) \circ_1 l)^{\s^{-1}} \:,
\end{equation*}
which agrees with the claim of Proposition \ref{model-classical-operad}.
\end{example}



\section{Derived degeneracy loci}\label{sec5}

This section, whose content consists of a few simple observations about the derived degeneracy loci and their classical counterparts,
concludes the paper.

\begin{lemma}\label{loci-class}
Classical induction along the map $\PP_1 \longrightarrow H^0(\PP_1^{\le m})$ sends a discrete Poisson algebra to its classical
$m$-degeneracy locus, see Definition \ref{def1} for the definition of the latter.
\end{lemma}

\begin{proof}
Recall from Proposition \ref{model-classical-operad} that $H^0(\PP_1^{\le m})$ is a factor of $\PP_1$ by an operadic ideal, which we 
denote $I^{(m+1)} \subset \PP_1$. In such a situation, classical induction along the factor map
\begin{equation*}
\PP_1 \longrightarrow \PP_1 / I^{(m+1)} \cong H^0(\PP_1^{\le m})
\end{equation*}
sends a discrete $\PP_1$-algebra $B$ into a factor algebra $B/I^{(m+1)}_B$, where by $I^{(m+1)}_B$ we denote an ideal in $B$ defined
as the image of the composite
\begin{equation*}
I^{(m+1)} (B) \hookrightarrow \PP_1(B) \longrightarrow B \:.
\end{equation*}
Recalling the description of $I^{(m+1)} \subset \PP_1$ from Proposition \ref{model-classical-operad}, we see that
$I^{(m+1)}_B \subset B$ is the ideal generated by $(m+1)$-th skew power of the Poisson bivector.
\end{proof}

\begin{definition}\label{der-degen-loci-def}
Let $A \in \Alg_{\PP_1}$ be a Poisson algebra. Its $m$-th derived degeneracy locus $\T^{\le m}(A)$ is defined as the derived induction
$\Ind^{\PP_1^{\le m}}_{\PP_1}(A) \in \Alg_{\PP_1^{\le m}}$ of $A$ along the map $\PP_1 \longrightarrow \PP_1^{\le m}$;
the operad $\PP_1^{\le m}$ and the map $\PP_1 \longrightarrow \PP_1^{\le m}$ were defined in Definition \ref{main-def}.
\end{definition}

\begin{proposition}\label{der-degen-loci-prop}
Let $A \in \Alg_{\PP_1}$ be a connective, i.e. sitting in non-positive cohomological degrees, Poisson algebra. Then its degeneracy locus
$\T^{\le m}(A)$ is also connective, and the canonical map $A \longrightarrow \T^{\le m}(A)$ induces an isomorphism
\begin{equation}\label{loci-nill-iso}
\ind^{H^0(\PP_1^{\le m})}_{\PP_1}(H^0(A)) \stackrel{\cong}{\longrightarrow} H^0(\T^{\le m}(A))
\end{equation}
of discrete $H^0(\PP_1^{\le m})$-algebras. In other words, $H^0(\T^{\le m}(A))$ is given by the classical $m$-th degeneracy locus
of $H^0(A)$.
\end{proposition}

\begin{proof}
First, the fact that the functor $\Ind^{\PP_1^{\le m}}_{\PP_1}$ preserves connective objects follows from the fact that $\PP_1^{\le m}$ 
sits in non-positive cohomological degrees.

Let us denote by $^\text{cl}\alg_{\PP_1}$ the $1$-category of discrete Poisson algebras. Similarly, we denote by
$^\text{cl}\alg_{H^0(\PP_1^{\le m})}$ the $1$-category of discrete algebras over the operad $H^0(\PP_1^{\le m})$. There is a
commutative diagram
\[\xymatrix{
{^\text{cl}\alg_{H^0(\PP_1^{\le m})}} \ar[r]\ar[d] & \Alg^{\le 0}_{\PP^{\le m}_1} \ar[r] & \Alg^{\le 0}_{\PP_1} \\
{^\text{cl}\alg_{\PP_1}} \ar[rr] && \Alg^{\le 0}_{\PP_1} \ar@{}[u]|\inveq &,
}\]
whose arrows are the forgetful functors. Passing to left adjoints and using Lemma \ref{loci-class} to describe left adjoint to the 
left vertical arrow, we obtain a commutative diagram, whose commutativity is given by an isomorphism of the form (\ref{loci-nill-iso}).
\end{proof}



\bibliographystyle{IEEEtran}
\bibliography{refs}{}

\end{document}